\documentclass[10pt]{amsart}

\usepackage{amsmath, amssymb, euscript,  enumerate}

\newcommand{\R}{\mathbb R}

\newcommand{\B}{\mathcal B}

\newcommand{\Pc}{\mathcal P}
\newcommand{\T}{\mathcal T}
\newcommand{\e}{\varepsilon}
\newcommand{\te}{\theta}

\newcommand{\al}{\alpha}

\newcommand{\la}{\lambda}
\newcommand{\vp}{\varphi}

\newcommand{\D}{\nabla}
\newcommand{\ra}{\rightarrow}
\newcommand{\La}{\triangle}

\newtheorem{thm}{Theorem}[section]
\newtheorem{lem}[thm]{Lemma}
\newtheorem{prop}[thm]{Proposition}
\newtheorem{corollary}[thm]{Corollary}
\newtheorem{definition}[thm]{Definition}

\theoremstyle{remark}

\numberwithin{equation}{section}

\begin{document}

\title{Fully Degenerate Monge Amp\'ere Equations}

\author{Panagiota Daskalopoulos}

\address{
Department of Mathematics, Columbia University, New  York,
USA
}
\author{Ki-ahm Lee}

\address{
Department of Mathematics, Seoul National University, Seoul, Korea}
\begin{abstract}
In this paper, we consider the following nonlinear eigenvalue
problem  for the Monge-Amp\'ere equation:  find a non-negative weakly convex
classical solution $f$  satisfying
\begin{equation*}
\begin{cases}
\det D^2 f=f^p  \quad &\text{in $\Omega$}\\
f=\vp  \quad &\text{on $\partial\Omega$}
\end{cases}
\end{equation*}
for a strictly convex smooth domain $\Omega\subset\R^2$ and $0<p<2$.
When $\{ f=0\}$ contains a convex domain, we find a classical solution
which is smooth on $\overline{\{f>0\}}$ and whose free
boundary $\partial\{f=0\}$ is also smooth.
\end{abstract}
\maketitle
%% keywords here, in the form: keyword \sep 




%\section{}
%\label{}
\section{Introduction}
\label{intro}
We consider in this  work the following nonlinear eigenvalue
problem  for the Monge-Amp\'ere equation:  find a non-negative weakly convex
classical solution $f$  satisfying
\begin{equation}\label{eqn-f}
\begin{cases}
\det D^2 f=f^p  \quad &\text{in $\Omega$}\\
f=\vp  \quad &\text{on $\partial\Omega$}
\end{cases}
\tag{MA}
\end{equation}
for a strictly convex bounded smooth domain $\Omega\subset\R^2$,
with $\vp>0$ on $\partial\Omega$ and smooth, and $0<p<2$.

The study of problem (MA) is motivated by the general Minkowski problem
in differential geometry,  asking to find the manifold whose Gauss
curvature has been prescribed. More generally,  the Gauss curvature
itself may depend on the graph $z=f(P)$ of the manifold, namely
$$\det D^2 f(P)=h(P,f(P),\D f(P)).$$
For a  positive bounded $h$, this problem  has been discussed by many
authors  and the   $C^{1,1}$-regularity
of $f$ has been  established (c.f. in  \cite{GT}).   When $h$ is allowed to be zero, $f$ 
is not always a  $C^{1,1}$ function, as it will be discussed in the sequel. The regularity of $f$
is an  open problem (c.f. Aubin \cite{Au}). 

One  of the interesting 
cases is when $h=0$ on the vanishing set of $f$, especially when  $h=f^p$, as  in problem \eqref{eqn-f}. 
For   $p=1$, this problem  corresponds  to an eigenvalue problem
describing the asymptotic
behavior, as $t \to T$,  of the parabolic Monge-Amp\'ere equation
\begin{equation*}
\begin{cases}
f_t=\det D^2 f\quad &\text{in $\Omega \times (0,T)$}\\
f=\eta(t)\, \vp\quad &\text{on $\partial\Omega$}
\end{cases}
\end{equation*}
where $\eta(t)=1/(T -  t)$. 

For  $f<0$, $f=0$ on $\partial\Omega$ and  $h=(-f)^{-(n+2)}$ in $\Omega$,  problem \eqref{eqn-f} 
was  considered by Cheng  and Yau in  \cite{CY}. When,  $h=(-\lambda f)^{n}$ problem 
\eqref{eqn-f} corresponds to the eigenvalue problem for the
 concave operator $(\det D^{2}f)^{\frac{1}{n}}$  and has been studied  in 
\cite{Li} . The exponential nonlinearity, $h=e^{-2f}$ has been studied by Cheng and Yau in  \cite {CY2}. 
Equation $\det D^2 f =h$  with a degenerate source term $h$  has been studied at \cite{G}.
The limiting case $p\ra 0^+$, $f(x)^p\ra \chi_{\{f>0\}}$  
was  considered by
O. Savin, \cite{S} as  the obstacle problem for Monge-Amp\'ere equation, 
where the obstacle stays below the graph of $f$. The second author also
considered  the case
where  the obstacle stays above the graph of $f$. 

Since $f(x)^p\ra \chi_{\{f>0\}}$ as $p\ra 0^+$,
(\ref{eqn-f}) corresponds to  a perturbation problem for the obstacle problem
$$\det D^2 f=\chi_{\{ f >0 \}}$$ and $f\geq 0$ in $\Omega$.

\medskip

Depending on the boundary values $\vp$ one of the three
possibilities  may occur in (MA):

\begin{enumerate}[i.]
\item {\em $f>0$ in $\Omega$}: the equation (MA)
      is then strictly elliptic and by the  regularity theory
      of fully-nonlinear equations, $f$ is $C^\infty$ smooth
      in $\Omega$ (cf. \cite{CC}).
\item {\em $f \equiv 0$ on a convex sub-domain $\Lambda(f) \subset \Omega$}:
       equation (MA) becomes degenerate
       on $\Lambda(f)$ and  $\Gamma(f)=\partial \Lambda(f)$ is the free-boundary
       associated to this problem. The function $f$ is $C^\infty$ smooth
      on $\Omega (f)=\Omega\setminus \Lambda(f)$ (cf. \cite{CC}). 
      The optimal regularity of $f$ up to the interface will be discussed in
      this work. 
\item {\em $f(P_0)=0$ at a single point $P_0 \in \Omega$ and $f>0$,
      on $\Omega\setminus \{P_0\}$}: equation becomes
      degenerate at the point  $P_0$.  The function $f$ is $C^\infty$
      smooth on $\Omega\setminus \{ P_0\}$ (cf. \cite{CC}).
     However, the regularity of $f$ at $P_0$ is an open question. \end{enumerate}

\medskip

We will restrict our attention from now on to the case (ii) above, where the solution $f$
of \eqref{eqn-f}  vanishes on a domain $\Lambda (f)$. 
 
 By looking at
radial solutions $z=f(r)$ on $\Omega = B_2(0)$  which vanish on $B_1(0)$,
we find that the expected behavior of $f$ near the
interface $r=1$ is $f(r) \sim (r-1)_+^q$, with $q$ given in
terms of $p$ by $q = \frac{3}{2-p}.$
This motivates the introduction of the {\em pressure} function
\begin{equation}\label{eqn-prii}
g= q^{\frac 23} \, f^{\frac 1q} , \qquad q = \frac{3}{2-p}.
\end{equation}
%Let us define
%$$\Omega(g) = \{ \, x \in \Omega: \,\, g >0 \, \}.$$
A direct calculation shows that $g$ satisfies the problem 
\begin{equation}\label{eqn-g}
\begin{cases}
g\, \det D^2 g + \theta \,  (g_y^2 g_{xx} - 2 g_x g_y 
g_{xy} + g_x^2 g_{yy}) =\chi_{\{ \, g >0 \, \}} \quad &\text{in} \, \Omega\\
g=\bar\vp  \quad &\text{on} \, \partial \Omega
\end{cases} \tag{MAP}
\end{equation}
with
$$\theta  = \frac{1+p}{2-p}$$ and $\bar\vp(x)=q^{\frac 23}
\, \vp^{\frac 1q}$. One observes that $ \theta >0$ iff $p <2$ which explains
our assumption on $p$. 

A similar concept  of pressure plays an important role in obtaining the 
optimal regularity of solutions to another degenerate equation, this time parabolic,  the {\it porous medium equation}, namely the flow of a density function $f$ of a gas through a porous medium given  by 
\begin{equation}\label{eqn-pme}
f_t=\La f^m  \quad \mbox{on $\R^n$}.
\tag{PME}
\end{equation}
The corresponding pressure $g=f^{m-1}$ of the gas satisfies 
\begin{equation}
g_t=(m-1)g\, \La g+|\D g|^2, \qquad \mbox{on $\R^n$}.
\end{equation}
The pressure $g$ is more natural in terms of the regularity. For a classical solution, the expanding speed of the free boundary $\partial\Omega (g)=\partial\{g>0\}$ is $|\D g|$. 
If we observe that the free boundary expands with finite non-degenerate speed, $g$ grows linearly away from the free boundary 
       $\partial\Omega (g)$,  while the density $f$ grows like a H\"older function whose H\"older coeffcient depends on $m$, \cite{CVW}. The pressure $g$ is a kind of normalization of $f$. Then,  $g$ is Lipschitz on  $\R^n$ \cite{CVW}, \cite{CW}  and smooth on $\overline{\Omega (g)}$ 
\cite{DH1}, \cite{Ko}.

A pressure-like function $g=\sqrt{2f}$ for the  parabolic Monge Ampr\'e equation
\begin{equation}
f_t=\det D^2 f
\end{equation}
has also been shown to be Lipschitz globally and smooth on $\overline{\Omega (f)}$ in \cite{DH2}, \cite{DL1} and \cite{DL3}.

\medskip

Let us now turn our attention back to equation   \eqref{eqn-f}. 
Our  objective in this work is to establish  the existence of a classical solution $f$ of 
the problem    \eqref{eqn-f},  when the  boundary data $\varphi$ is such that the solution $f$ vanishes on a region $\Lambda(f) \subset \Omega$ and   therefore  the equation becomes degenerate near the interface $\Gamma(f)  = \partial \Lambda(f)$. The concept of a classical solution  will be discussed  in section \ref{sec-classic1}.  To guarantee that such vanishing region exists, we  
assume that there is a classical super-solution $\psi$ of \eqref{eqn-f} vanishing on  a non-empty domain $\Lambda(\psi) \subset \Omega$. In section \ref{sec-classic2} we will actually present an example which shows that this
is indeed possible.

\begin{thm}\label{theorem-main} 
 Assume  that $\Omega\subset\R^2$ is a strictly convex bounded smooth domain and let  $\vp \in C^2(\bar \Omega)$, $\vp>0$ on $\partial\Omega$  and $0<p<2$. Assume that  there is a classical super-solution $\psi$ of \eqref{eqn-f} vanishing on  a non-empty domain $\Lambda(\psi) \subset \Omega$. Then, there is a classical solution of \eqref{eqn-f} and its pressure $g$, given in terms of $f$ by  \eqref{eqn-prii},  is $C^\infty$ smooth on $\Omega(g)$ up to the interface  $\Gamma$. Consequently, $f$ enjoys the optimal regularity $f \in C^{k,\alpha}$ with 
$k=[ \frac 3{2-p} ]$, $\alpha = \frac 3{2-p} -k$,  
( $C^{k-1,1}$, if $k:=\frac 3{2-p}$ is an integer)  and the interface $\Gamma(g)$  is $C^\infty$ smooth.
\end{thm}

%We finish this Introduction with a remark about  uniqueness of classical solutions  and their connection with viscosity solutions. 

%\begin{remark} One may ask whether classical solutions of \eqref{eqn-f} which satisfy the regularity conditions of Theorem \ref{theorem-main} are unique. In \cite{DL4} the authors
%have recently shown   that a classical solution of equation  \eqref{eqn-f} which satisfies  the regularity conditions of Theorem \ref{theorem-main} coincides with the minimal viscosity solution. Hence, the uniqueness of such solutions follows. The approach  in \cite{DL4} is via parabolic theory. The minimal viscosity solution is
%obtained as the steady state limit of the associated  evolution equation of Monge-
% Amp\'ere type. 
%\end{remark}

\smallskip

A brief outline of the paper is as follows: in section \ref{sec-classical}   the concept of   
classical solutions     of \eqref{eqn-f}  is introduced  and  the proof of its existence via the method
of continuity is outlined.   Section \ref{sec-estimates}  will be devoted to the derivation  of sharp a'priori derivative
estimates for classical solutions of equation \eqref{eqn-g}. These estimates play  crucial 
role in establishing the $C^{2,\alpha}_s$ regularity of classical solutions of
\eqref{eqn-g}  which will
be shown in section \ref{sec-c2a}  (see in section \ref{subsec-linearized} for the definition of this space). Based on the estimates in section \ref{sec-c2a}, we will conclude,  in
section \ref{sec-stability},   the proof of the existence of a $C^{2,\alpha}_s$  up to the interface 
solution $g$  of \eqref{eqn-g} , via the method of continuity. We will also show that  the pressure $g$ is $C^\infty$ smooth up to the interface.

{\bf Notation: }
\begin{itemize}
\item $\Omega \subset \R^2$ denotes a {\em strictly convex}  bounded smooth domain in $\R^2$.
\item $\vp$ denotes a smooth {\em strictly positive}  function defined on $\partial\Omega$.
\item For any $g \geq 0$ on $\Omega$, we denote 
$$\Omega (g)=\{x \in \Omega \, |\, g(x)>0\}, \quad  \Lambda(g) = \{x \in \Omega \, |\, g(x)=0\}$$
 and   $\Gamma (g)=\partial \Lambda (g)$.
\item $ds^2$ denotes the singular metric defined in section \ref{subsec-linearized}.
\item $\|g\|_{C^2_{\partial \Omega}} = \max_{\partial \Omega} (|D_{ij} g|  + |D_i
g| + g)$.
\item $C^{2,\alpha}_{s}(\Omega)$ will be defined in section  \ref{subsec-linearized}
and $C^{k,2+\alpha}_{s}(\Omega)$ will be defined in section \ref{sec-stability}.
\item $\nu$, $\tau$ denote the outward normal and tangential directions to the
level sets of a function $g$. 
\item $g_{\nu}, g_{\tau}, g_{\nu\nu}, g_{\nu\tau}, g_{\tau\tau}$ denote the derivatives of $g$ with respect
to $\nu$, $\tau$. 
%\item $\T$ denotes  the set of nonnegative convex functions defined at Definition \ref{def-vis}.
%\item $f^{*}$ denotes  the convex envelope of $f$, namely   the supremum of all possible linear functions below $f$.
\end{itemize}

\bigskip

%%%%%%%%%%%%%%%%%%%%%%%%%%%%%%%%%%%%%%%%%%%%%%%%%%%%%%%%%%%%%%%%%%%%%%%%%%%%%%%%%
%%%%%%%%%%%%%%%%%%%%%%%%%%%%%%%%%%%%%%%%%%%%%%%%%%%%%%%%%%%%%%%%%%%%%%%%%%%%%%%%%

\section{Classical  solutions and the method of continuity}\label{sec-classical}

In this section we will  define the   concept of  a classical
solution of equation (\ref{eqn-f}) [ resp. of \eqref{eqn-g}]  and  sketch the proof of  its existence   via a method of
continuity.

\subsection{The concept of classical solutions and the comparison principle}\label{sec-classic1}

We consider the following generalization of  equation (\ref{eqn-g}),
namely  
\begin{equation}\label{eqn-g2}
g\, \det D^2 g + \theta  \, (g_y^2\, g_{xx} - 2\, g_x g_y \,
g_{xy} + g_x^2\, g_{yy}) =h\, \chi_{\{ \, g >0 \, \}}\tag{MAPh}
\end{equation}
%\begin{equation}\label{eqn-g2}
%\begin{cases}
%g\, \det D^2 g + \theta  \, (g_y^2\, g_{xx} - 2\, g_x g_y \,
%g_{xy} + g_x^2\, g_{yy}) =h\, \chi_{\{ \, g >0 \, \}} \quad &\text{in} \, \Omega\\
%g=\bar\vp  \quad &\text{on} \, \partial \Omega.
%\end{cases}\tag{MAPh}
%\end{equation}
where  $h \in C^2(\Omega)$ and satisfies the bounds
\begin{equation}\label{eqn-lambda}
 0< \lambda < h < \lambda^{-1} < \infty
 \end{equation}
for some constant $\lambda >0$. 

We recall the notation $\Omega (g)=\{x \in \Omega \, |\, g(x)>0\}$ and $\Lambda(g) = \Omega \setminus
\Omega (g)$. 
On the free-boundary $\Gamma(g):= \partial \Lambda (g)$, where $g=0$,   we
then have, from \eqref{eqn-g2},
$$\theta\, (g_{x}^2\,  g_{yy}-2g_x g_y g_{xy}+g_{y}^2\, g_{yy})=\theta\, g_{\nu}^2 g_{\tau\tau}
=\theta \, g_{\nu}^3  \kappa = h$$ where $\nu$ and $\tau$ are inward
normal and tangential unit directions to $\Gamma(g)$
respectively and where $\kappa=g_{\tau\tau}/g_\nu$ denotes the   curvature of $\Gamma(g)$. 

More generally, denote by $\nu$, $\tau$ the outward normal and tangential directions to the
level sets of the function $g$. 

\begin{definition}
We say that  $g \in C^2_s(\overline{ \Omega(g)})$ iff 
$$g, \, g_\nu, \, g_{\tau}, \, g\, g_{\nu\nu}, \, \sqrt{g}\, g_{\nu\tau}, \, g_{\tau\tau}$$
extend  continuously up to $\overline {\Omega(g)}$ and are bounded on $\Omega(g)$. 
\end{definition}

Define the non-linear operator
\begin{equation}\label{eqn-op}
 \mathcal{P} [g] :=g\det D^2 g+\theta(g_{x}^2
g_{yy}-2g_xg_yg_{xy}+g_{y}^2g_{yy}).
\end{equation}

\begin{definition} Assume that $g\in C^{0,1}(\Omega)\cap
C^{2}_{s}(\overline{\Omega (g)}  )$ and that  $f=\left(
q^{-\frac 23} \, g \right)^q$, $q=\frac{3}{2-p}$  is convex in $\Omega$. 
The function $g$ is called a classical super-solution (sub-solution) of  equation \eqref{eqn-g2} if
\begin{equation}
\begin{cases}
\mathcal{P}[g]  &\leq h \, ( \, \ge h  ) \quad
\text{in $\Omega(g)$}\cr
\Pc [g]=\theta g_{\nu}^3 \, \kappa &\leq h\, (  \ge h )  \quad
\text{on $\Gamma(g)$}.
\end{cases}
\label{eqn-c-g}
\end{equation}
The function $g$ is called a  classical solution if it is both a classical super-solution  and sub-solution.
\end{definition}

If $g$ satisfies \eqref{eqn-g2},  then the  corresponding convex function $f=\left(
q^{-\frac 23} \, g \right)^q$, with $q=\frac{3}{2-p}$, satisfies the equation
\begin{equation}\label{eqn-f2}
\det D^2 f= h  f^p.
\tag{MAh}
\end{equation}
%\begin{equation}\label{eqn-f2}
%\begin{cases}
%\det D^2 f=\, h \, f^p  \quad &\text{in $\Omega$}\\
%f=\vp  \quad &\text{on $\partial\Omega$}
%\end{cases}
%\tag{MAh}
%\end{equation}

\begin{definition}
 A convex function $f$ is called a classical super-solution, sub-solution,  or solution of \eqref{eqn-f2} if
      the correponding pressure $g$ belongs to  $C^{0,1}(\Omega)\cap
C^{2}_{s}(\overline{\Omega (g)})$ and is a  classical  super-solution, sub-solution, or solution of \eqref{eqn-g2} respectively.
\end{definition}

\begin{lem}  Let $g_1$ be a classical  super-solution and $g_2$ be a classical sub-solution of
\eqref{eqn-g2} such that $g_2 < g_1$ on $\partial\Omega$.
Assuming  that $\Omega(g_2)\subset\Omega (g_1)$, we have $g_2\leq g_1$
in $\Omega$.
\end{lem}
\begin{proof}
Choose $\e>0$ sufficiently small so that  $g^{\e}_2:=(1+\e)g_2<g_1$ on
$\partial\Omega$ and $g^{\e}_2$ is   a strict sub-solution of
(\ref{eqn-g2}). We claim that $g^{\e}_2:=(1+\e)\, g_2 \leq g_1$ in $\Omega$. 

Indeed, let us assume that $g^{\e}_2$ touches $g_1$ at a
point $P_0$.  If  $P_0\in\Omega(g^{\e}_2)$, then 
$$h(P_0)\geq
\mathcal{P} [g_1](P_0)\geq \mathcal{P} [g^{\e}_2](P_0)>h(P_0)$$  which is a
contradiction.

Hence, we may assume that 
$P_0\in\partial\Omega(g^{\e}_2)$. Clearly  $\partial\Omega
(g^{\e}_2)$ will also touch $\partial\Omega (g_1)$ at $P_0$. Then, at $P_0$, we  have $(g^{\e}_2)_{\nu}\leq (g_1)_{\nu}$ and
$\kappa_2\leq\kappa_1$,  where $\kappa_1, \kappa_2$ denote the  curvatures of
$\partial\Omega (g_1)$, $\partial\Omega
(g^{\e}_2)$ respectively. Thus at $P_0$ 
$$(g_1)_{\tau\tau}=(g_1)_{\nu} \, \kappa_1\geq (g^{\e}_2)_{\nu}\, \kappa_2
=(g^{\e}_2)_{\tau\tau}$$
and then
$$\frac{h}{\theta}=(g_1)_{\nu}^2(g_1)_{\tau\tau}\geq (g^{\e}_2)_{\nu}^2
(g^{\e}_2)_{\tau\tau} =\frac{(1+\e)h}{\theta}$$ which is a
contradiction. This finishes the proof of our claim. 

Since  $(1+\e)\, g_2 = g^{\e}_2\leq g_1$ for
any small $\e>0$, letting  $\e\ra 0$ we conclude that  $g_2(P)\leq g_1(P)$.
\end{proof}

\begin{thm}[Comparison Principle for Classical Solutions] 
\label{theorem-c-comparison} Let $g_1$ be a classical super-solution and $g_2$ be a classical sub-solution of
\eqref{eqn-g2} such that $g_2\leq g_1$ on $\partial\Omega$. Assuming  that $\Omega(g_2)\subset\Omega (g_1)$, we have $g_2\leq g_1$
in $\Omega$.
\end{thm}
\begin{proof}
The function $g^{\e}_2=(1+\delta_{\e})\, (g_2-\e)_+$ is also a
strict sub-solution of (\ref{eqn-c-g}) such that $g^{\e}_2< g_1$
on $\partial\Omega$, for a small $\delta_{\e}$ depending $\e$. For
large $\e>0$, $\Omega(g^{\e}_2)\subset\Omega (g_1)$ and let $\e$
decrease to zero. If $g_2\nleq g_1$ in $\Omega$, there is a
positive $\e>0$ such that $g^{\e}_2$
 touches $g_1$ at a point $P_0\in\overline{\Omega(g^{\e}_2)}$.
The same argument  as in the lemma above shows that   $P_0$ can not be a point
 in $\Omega (g^{\e}_2)$, since $g^{\e}_2$ is a strict sub-solution.
Also  $P_0 \notin\partial\Omega(g^{\e}_2)$; 
otherwise $\partial\Omega(g^{\e}_2)$ would  touch $\partial\Omega (g_1)$, 
  for $\e>0$, which leads to  contradiction similarly as  in the proof of the  previous lemma.
\end{proof}

\subsection{The linearized operator near the free-boundary and sharp a' priori estimates}\label{subsec-linearized}
In sub-section \ref{sec-classic2}  we will outline the proof of the  existence of a classical solution of problem \eqref{eqn-g}
via the method of continuity. 
Our approach  relies on the observation that one can  obtain sharp
a priori estimates for classical solutions $g$ of the degenerate equation \eqref{eqn-g2}
if one scales the estimates according to the natural   singular metric corresponding to problem .

To illustrate this better, assume that $g$ is a classical solution of equation  \eqref{eqn-g2} and 
that $P_0 \in \Gamma(g)$ is a free-boundary point.
We will show in section 3, that $g$ satisfies the a'priori bounds \eqref{eqn-grad0} and
\eqref{eqn-matrix30}  near the free-boundary, which in particular imply the bound 
$$ c < |Dg| \leq c^{-1}$$ for some $c >0$. We may assume,  without loss of generality,  that 
$g_x >0$, $g_y=0$ at $P_0$
so that it is possible  to solve  the equation
$z=g(x,y)$ near $P_0$ with respect to $x$ yielding  to a map
$x=q(z,y)$ defined for all $(z,y)$ sufficiently close to $Q_0=(0,y_0)$.
The function $q$ satisfies the equation  
\begin{equation}\label{eqn-q10}
\frac { - z\, {\det} D^2  q  +  \theta \, q_z \, q_{yy} }
{q_z^4} = - H (z,y) 
\end{equation}
where $H(z,y) := h(x,y)$, $x=q(x,y)$. 
Based on the a'priori estimates,   we will show in section 4 that the linearized operator 
of equation \eqref{eqn-q10}  near a  function  $q$ satisfying  
the  bounds \eqref{eqn-matrix5} and \eqref{eqn-coeffb}   is of the form
\begin{equation}\label{eqn-linear}
L(\tilde q) = z\, \alpha_{11} \, \tilde q_{zz} + 2 \sqrt z \al_{12} \, \tilde q_{zy} + \alpha_{22}\, \tilde q_{yy} + b \, \tilde q_z  + c\, \tilde q 
\end{equation}
with $(\alpha_{ij})$ strictly positive and  $b \geq \nu >0$.

To apply the method of continuity  one needs to establish  sharp a' priori estimates for linear degenerate equations of the  form \eqref{eqn-linear}. These estimates become optimal when scaled
according to the singular metric 
\begin{equation}\label{eqn-metric}
ds^2 = \frac{dz^2}{z} + dy^2 \tag{SM}
\end{equation}
which is the natural metric  corresponding   this problem. 

Denote   by  $\mathcal B_\eta$ the box $
\mathcal B_\eta =
\{ \, 0\leq z \leq \eta^2
, \,\, |y-y_0| \leq \eta \,\, \} $ and for any two points
$Q_1=(z_1,y_1)$ and $Q_2=(z_2,y_2)$ in ${\mathcal B}_\eta$,  by $s$ the distance
function 
\begin{equation}\label{dist}
s(Q_1,Q_2) = |\sqrt{z_1}-\sqrt{z_2}  | + |y_1-y_2| \tag{DF}
\end{equation}
 with respect  to the singular
metric $d s^2$.  Let     $C^\alpha_s({\mathcal B}_\eta)$ be  the space of 
all H\"older continuous functions 
on ${\mathcal B}_\eta$ with respect to the distance function $s$. Suppose that the
function $q$ belongs to the class $C^\alpha_s({\mathcal B}_\eta)$
and has continuous derivatives $q_t, q_z, q_y, q_{zz}, q_{zy}, q_{yy}$
in the interior of ${\mathcal B}_\eta$,  and that 
\begin{equation}\label{space}
q,\, q_z, \,  q_y, \,  z\, q_{zz}, \, \sqrt z\, q_{zy}, \, q_{yy} \in
C^\alpha_s({\mathcal B}_\eta)
\end{equation}
extend continuously up to the boundary, and the extensions are H\"older continuous on ${\mathcal B}_\eta$ of class $
C^\alpha_s({\mathcal B}_\eta)$ as before. 
We denote by  $C^{2+\alpha}_s({\mathcal B}_\eta)$  the
Banach space of all such functions with norm
$$\aligned
\|q\|_{C^{2+\alpha}_s({\mathcal B}_\eta)} &= \|f\|_{C^\alpha_s({\mathcal B}_\eta)} + \|Dq\|_{C^\alpha_s({\mathcal B}_\eta)} 
+ \|q_t\|_{C^\alpha_s({\mathcal B}_\eta)}\\
& + \|z \, q_{zz}\|_{ C^\alpha_s({\mathcal B}_\eta)} +
\|\sqrt z\, q_{zy}\|_{ C^\alpha_s({\mathcal B}_\eta)} +
\|q_{yy}\|_{ C^\alpha_s({\mathcal B}_\eta)}\ .\endaligned
$$

\begin{definition}\label{def-c2a}
We say that $g \in C^{2,\alpha}_s(\overline{\Omega(g)})$ if $g$ is of class $C^{2,\alpha}$ in the
interior of $\Omega(g)$ and its transformation $q \in C^{2,\alpha}_s({\mathcal B}_\eta)$ near any free-boundary point $P_0$. We denote by $\|g\|_{C^{2,\alpha}_s}$ the corresponding norm. 
\end{definition}

The following  result follows as an easy modification of Theorem 5.1 in \cite{DH2}.

\smallskip

\noindent{\bf Theorem [DH] } (Schauder estimate). {\em Assume that the coefficients
of the operator $L$ given by \eqref{eqn-linear} belong to the class $C^\alpha_s({\mathcal B}_\eta)$, for some $\eta >0$,  and $(a_{ij})$ is strictly positive. Then,
for any $r < \eta$}
$$\|\tilde q\|_{C^{2,\alpha}_s({\mathcal B}_{r})} \leq C \, \left (  \, \|\tilde q\|_{C^{\circ}({\mathcal B}_\eta)} +
\| h \|_{C^\alpha_s({\mathcal B}_\eta)}\,  \right )$$
{\em for all smooth functions $\tilde q$ on ${\mathcal B}_\eta$ for which $L\tilde q=h$. }

\smallskip

The following result was shown in \cite{DL2}. 

\smallskip

\noindent{\bf Theorem [DL] } (H\"older regularity). 
{\em Assume that the coefficients of the operator 
$L$ given by \eqref{eqn-linear}  are bounded measurable  on ${\mathcal B}_\eta$, $\eta >0$,  with
$(a_{ij})$ strictly positive and $b \geq \nu >0$. 
Set $ d\mu =  x^{\frac \nu2-1}\, dx\, dy$. Then, 
there exist a number $0< \alpha <1$ so that, for any $r<\eta/2$ }
$$\|\tilde q\|_{C^{\alpha}_s({\mathcal B}_{r})} \le C \,  \left ( \|\tilde q\|_{ C^{\circ}
({\mathcal B}_\eta) } + ( \int_{ {\mathcal B}_\eta}  h^2 \, d\mu )^{1/2} \right )$$
{\em for all smooth functions $\tilde q$ on ${\mathcal B}_\eta$ for which $L\tilde q=h$. }

Based on Theorem [DL] and the sharp a priori bounds Theorem \ref{theorem-c-2},  the following a priori estimate will be shown in  section 4. 

\smallskip

\begin{thm} [$C^{2,\alpha}_s$-estimate]  \label{theorem-c2a2} Assume that $g \in C^4(\overline {\Omega (g)})$ is a 
classical solution of problem  \eqref{eqn-g},  with  $0 < p < 2$,  and that $B_\rho(0) \subset \Lambda(g)$. Then, there exists a constant $C=C(\|\bar \vp \|_{C^2_{\partial \Omega}}, \theta,  \rho\, )>0$ such that $\|g \|_{C^{2,\alpha}_s(\overline {\Omega(g)}) } \leq C$.
\end{thm} 

\smallskip

Based on  Theorem  [DH]  and Theorem  \ref{theorem-c2a2},   a $C^{2,\alpha}_s$ solution of
the problem \eqref{eqn-g}  will be  constructed  via the method of continuity. It follows from Theorem [DH] and an inductive argument  that  the pressure $g$ is $C^\infty$ smooth up to the interface
$\Gamma(g)$,  which readily implies  that the interface is smooth (c.f. section  4).

\subsection{Existence of solutions via the method of continuity}\label{sec-classic2}
We will now outline the basic steps of the proof  of the existence of a 
classical solution  of \eqref{eqn-g} via the method of continuity. The proofs of these steps will be given in the
following sections. 

According to our assumption in Theorem \ref{theorem-main},  there exists  a super-solution $\psi$ of \eqref{eqn-f}, i.e., $\psi$ satisfies 
$$\det (D^2 \psi) \leq \psi^p, \qquad \mbox{in} \,\, \Omega$$
which vanishes 
on a non-empty domain $\Lambda(\psi) \subset \Omega$.
We define   $$\overline{h} : =\frac{\det (D^2 \psi)}{ \psi^p}\leq 1.$$
Before we proceed with the outline of the method of continuity, let us give an example
which shows that there exist boundary values  $\phi$ for which such a super-solution
can be found. 

\smallskip

\noindent{\bf Example.} Set   
$$\psi_1(P)= c_1(|P|^2-\rho^2)_+^q, \qquad q= \frac 3{2-p}$$
and pick a $c_1>0$ so that $$ 2\lambda<\frac{\det(D^2 \psi_1(P))}{\psi_1(P)^p}
 <\frac12$$ for some $\lambda \in (0,1)$ in $\Omega(\psi_1)$.  When the boundary data $\vp$  in  \eqref{eqn-f}  is such that
 $$\vp \geq \psi_1, \qquad \mbox{on} \,\, \partial\Omega$$ we can modify $\psi_1$
 to a convex function $\psi(P)$, keeping the decay rate to zero
 on $\partial\Omega (\psi)$, so that $\psi(P)=\vp(P)$ on $\partial\Omega$ and 
\begin{equation}\label{eqn-oh}
\lambda<\frac{\det(D^2\psi(P))}{\psi(P)^p}: =\overline{h}(P)
 <1.
 \end{equation}
Hence, $\psi$ is the desired super-solution. 
\smallskip

Going back to the method of continuity, we consider the following boundary value problems depending on  a parameter $t\in[0,1]$:
\begin{equation}
\begin{cases}
\det (D^2 f(P))=((1-t)\overline{h}+t)f^p\quad &\text{in $\Omega$}\\
f=\vp\quad &\text{on $\partial\Omega$.}
\end{cases}
\tag{MAt}
\label{eqn-mat}
\end{equation}
Set $h_t :=(1-t)\overline{h}+t$ and observe that 
$$\lambda  < h_t \leq  1$$
since $\overline{h}$ satisfies \eqref{eqn-oh}. Hence, $h_t$ satisfies condition \eqref{eqn-lambda}. Also, 
since $h_t \geq \bar h$,  a classical solution $f(P ;t)$ of \eqref{eqn-mat} is a sub-solution
of 
\begin{equation}\label{eqn-hhh}
\begin{cases}
\det (D^2 f(P)) =  \overline{h}\,  f^p\quad &\text{in $\Omega$}\\
f=\vp\quad &\text{on $\partial\Omega$}
\end{cases}
\end{equation}
while the given  $\psi(x)$ is a super-solution of  \eqref{eqn-hhh}.
Hence,  by the comparison  lemma \ref{theorem-c-comparison},  if $\{\psi(P) =0\} \subset \{f(P;t)=0\}$, then
$f(P ;t)  \leq \psi(P)$ in $\Omega$.
We are going to carry out the method of continuity starting with $f_0=\psi(x)$ at $t=0$, keeping 
$$ \{\psi(P) =0\} \subset \{f(P ;t) =0\}, \qquad   \mbox{for}   \,\,\,    0\leq t\leq 1 $$ so  that $f(P ;t)$ 
has  a non-empty  vanishing region $\Lambda(f(P ;t))$, for every $t \in [0,1]$. This justifies
our assumption (H-2) below. 

\medskip
Assume that $f$ is a classical solution of \eqref{eqn-mat} (we will drop the index $t$ on $f$ for the rest of the section). Then, the corresponding  pressure  function
$g$, defined in terms of $f$ by  (\ref{eqn-prii}), satisfies
\begin{equation}
\begin{cases}
g\det D^2
g+ \theta \, (g_{x}^2g_{yy}-2g_xg_yg_{xy}+g_{y}^2g_{xx})=h_t \quad
&\text{in $\Omega$}\\
g=\overline{\vp} \quad &\text{on $\partial\Omega$}
\end{cases}
\tag{MAPt}
\label{eqn-mapt}
\end{equation}
for $\overline{\vp}=q^{\frac23}\vp^{\frac1q}$.

\smallskip
We  make the  following assumptions:

\begin{enumerate}[(H-1)]

\item $\Omega \subset B_1(0)$.
\item $f$ and $g$ vanish on a non-empty sub-domain 
$\Lambda (f)=\Lambda (g) \subset
\subset \Omega$ and $$B_\rho(0)=\{ x \in \R^2: \,\, |x| < \rho \}
\subset \Lambda (f), \quad \mbox{for some} \,\,   \rho >0.$$

\item $f$  is strictly positive and strictly convex
 on $\Omega (f)=\{x\in \Omega|\, f>0\}.$

\item The pressure  $g$ satisfies  $g \in C^4(\overline{ \Omega(g)})$, i.e.,
in particular it is $C^4$-smooth up to $\partial \Omega (g)$.
\end{enumerate}

To simplify the notation,  we will set from now on
$$\|g\|_{C^2_{\partial \Omega}} = \max_{\partial \Omega} (|D_{ij} g|  + |D_i g| + g).$$
In the next section  we will establish sharp  a-priori bounds on the first and second
derivatives of the pressure  $g$ up to the interface $\partial
\Omega$, as stated in the sequel. 

%Let $\nu_{P_0}$ be an outward normal unit vector and $\tau_{P_0}$
%be a tangential unit vector at $P_0$ to the level set
%$\{x|g(P)\leq g(P_0)\}$.

%\begin{thm} \label{theorem-c-2}  Assume that $g \in C^4(\overline{\Omega(g)})$ is a classical
%solution of equation  \eqref{eqn-g2} in $\Omega$ 
%with  $0 < p < 2$ and $h \in C^2(\Omega)$ satisfying  \eqref{eqn-lambda}.  Assume in addition that    $g$ satisfies the assumptions (H-1)--(H-4).
%Define the matrix
%\begin{equation}\label{eqn-matrix22}
%\mathcal M= (\mu_{ij})  = \left (
%\begin{split}
%g\, g_{\nu\nu}+ \theta \, g_\nu^2 &  \quad \sqrt{g} \, g_{\nu\tau}\\
%\sqrt{g} \, g_{\nu\tau}  \,\,\,\, & \quad \,\,\, g_{\tau\tau}
%\end{split}\right )
%\end{equation}
%with $\nu$, $\tau$ denoting the outer normal and tangent
%direction to the level sets of $g$ respectively.
% Then,  there exists  $c=c(\|g\|_{C^2_{\partial \Omega}}, \|h\|_{C^2}, \theta, \lambda, \rho\, )>0$,   for which  the  bounds
%\begin{equation}\label{eqn-grad}  c < |Dg| \leq c^{-1}
%\end{equation}
%and
%\begin{equation}\label{eqn-matrix3}
% c\, |\xi|^2  \leq  \mu_{ij} \xi_i \, \xi_j  \leq c^{-1}\, |\xi |^2, \qquad \forall \xi \neq 0
% \end{equation}
%hold on $\overline{\Omega (g)}$.
%\end{thm}

\begin{thm}[$C^{2}_s$-estimate]  \label{theorem-c-2} Assume that $g$ is a classical
solution of equation  \eqref{eqn-g2} in $\Omega$ 
with  $0 < p < 2$  and   $h \in C^2(\Omega)$ satisfying \eqref{eqn-lambda}. 
Assume in addition that $g$ satisfies the assumptions (H-1)--(H-4).
Define the matrix
\begin{equation}\label{eqn-matrix20}
\mathcal M= (\mu_{ij})  = \left (
\begin{split}
g\, g_{\nu\nu}+ \theta \, g_\nu^2 &  \quad \sqrt{g} \, g_{\nu\tau}\\
\sqrt{g} \, g_{\nu\tau}  \,\,\,\, & \quad \,\,\, g_{\tau\tau}
\end{split}\right )
\end{equation}
with $\nu$, $\tau$ denoting the outer normal and tangent
direction to the level sets of $g$ respectively.
 Then,  there exists  $c=c(\|g\|_{C^2_{\partial \Omega}}, \|h\|_{C^2}, \theta,  \lambda, \rho\, )>0$,   for which  the  bounds
\begin{equation}\label{eqn-grad0}  c \leq |Dg| \leq c^{-1}
\end{equation}
and
\begin{equation}\label{eqn-matrix30}
 c\, |\xi|^2  \leq  \mu_{ij} \xi_i \, \xi_j  \leq c^{-1}\, |\xi |^2, \qquad \forall \xi \neq 0
 \end{equation}
hold on $\overline{\Omega (g)}$.
\end{thm}

Combining Theorem \ref{theorem-c-2} with the H\"older regularity Theorem [DL], we
will show in section  \ref{sec-c2a} the following a priori estimate. 

\begin{thm}[$C^{2,\alpha}_s$-estimate] \label{theorem-c-2a} Under the same conditions as in  Theorem \ref{theorem-c-2}, 
there is a uniform $0<\alpha<1$ and $C=C(\|g\|_{C^2_{\partial \Omega}}, \|h\|_{C^2},\theta,  \lambda, \rho\, )< \infty$, such that       $$
       \| g\|_{C^{2,\alpha}_s(  \overline{\Omega (g)})} \leq C.$$
       In addition the curvature $\kappa(g)$ of the free-boundary $\Gamma(g)$ is of class $C^{\alpha}$.
\end{thm}

The above result shows  that  the coefficients of the matrix  \eqref{eqn-matrix20} are  
uniformly H\"older. This will  be combined  in section  \ref{sec-c2a} with the Schauder estimate, Theorem [DH],  to 
obtain  the following regularity of $g$.
\begin{thm} [Higher regularity] \label{theorem-c-smooth}
 Under the same conditions as  in Theorem \ref{theorem-c-2} and the additional
 assumption that $h \in C^\infty(\Omega)$, the solution 
      $g$ of \eqref{eqn-g2} is smooth  on ${\Omega(g)}$ up to the interface $\Gamma(g)$ which means that for every positive integer, there exists $C_k=C(\|g\|_{C^2_{\partial \Omega}}, \|h\|_{C^{k+2}}, \theta,  \lambda, \rho, k\, )< \infty$ for which 
      $$\|g\|_{C^{k+2,\alpha}_{s}( \overline{\Omega (g)})} \leq C_k$$
       and   the curvature $\kappa(g)$ of $\Gamma  (g)$
      is $C^{k,\alpha}$. It follows that $g$ is $C^\infty$-smooth
      up to the interface $\Gamma(g)$ and that the interface is smooth. 
\end{thm}

To implement the method of continuity, we next set 
\begin{equation*}
I=\{t\in[0,1]|\> \text{(\ref{eqn-mapt}) has a classical
solution satisfying  (H-1)-(H-4)}\}. 
\label{eqn-i}
\end{equation*}

Clearly $I$ is
{\em  nonempty} since by the assumption of Theorem \ref{theorem-main} $\psi$ is a solution of (\ref{eqn-mapt}) for
$t=0$. The existence of classical solution of \eqref{eqn-g2} is
equivalent to  that $1\in I$. The method  of
continuity relies on showing that the  nonempty set $I$ is both open and
closed in $[0,1]$ in the relative topology, which means that $I=[0,1]$ and
hence $1\in I$.

The {\em closedness} of $I$ easily  follows from  Theorems
\ref{theorem-c-2} and \ref{theorem-c-2a}, as shown next. 

\begin{lem} The set $I$ is closed.
\end{lem}
\begin{proof}
Let   $\{t_k\}\subset I$ be a sequence converging to $t_0$. Then,  there is a
sequence of solutions $\{g_k\}$ of (MAPt), $t=t_k$, and their 
free-boundaries $\Gamma (g_k)$ which  have
uniform estimates depending only on the boundary data and the
domain $\Omega$. First we can extract a converging
subsequence of the free boundaries $\Gamma (g_{k_i})$ to
$\Gamma_0$ and, among them, extract converging
subsequence   $g_{k_{i_j}}$ converging to a function $g_0$. The 
 non-degeneracy estimate in \eqref{eqn-grad0} implies that 
$\Gamma_0=\Gamma(g_0)$  and the uniform
$C^{2,\alpha}_{s}$-estimate in Theorem \ref{theorem-c-2a} implies   that $g_0$ is   a solution of
\eqref{eqn-mapt} with $t=t_0$. Hence $t_0\in I$.
\end{proof}

The {\em  openness} of $I$ will be proved in Section \ref{sec-stability} through the stability in the parameter $t$, Theorem \ref{theorem-stability}, which  is similar to Theorem 8.5 in \cite{DH2}.

The method of continuity then implies   the following existence of classical solutions.

\begin{thm}[Existence of a classical solution] Under the assumptions of Theorem \ref{theorem-main},  there is a classical solution $g$ of (\ref{eqn-g}) which satisfies
the estimates in Theorems \ref{theorem-c-2}, \ref{theorem-c-2a}, and
\ref{theorem-c-smooth}.
\end{thm}

The rest of the paper will be devoted to the proof of Theorems \ref{theorem-c-2} - Theorem \ref{theorem-c-smooth} which,  in  particular, imply Theorem \ref{theorem-main}. 

%%%%%%%%%%%%%%%%%%%%%%%%%%%%%%%%%%%%%%%%%%%%%%%%%%%%%%%%%%%%%%%%%%%%%%%%%%%%%%%%%%%
%%%%%%%%%%%%%%%%%%%%%%%%%%%%%%%%%%%%%%%%%%%%%%%%%%%%%%%%%%%%%%%%%%%%%%%%%%%%%%%%%%%
\section{Optimal  Estimates}\label{sec-estimates}
In this section we are going to prove the  optimal a'priori estimates stated in  Theorem \ref{theorem-c-2}. We will assume,  throughout this section, that $g \in C^4(\overline{\Omega(g)})$
is a classical solution of equation \eqref{eqn-g2} in $\Omega$ with  $0 < p < 2$ and  $h \in C^2(\Omega)$ satisfying  \eqref{eqn-lambda}. 
In  addition, we will assume  that $g$ satisfies the assumptions (H-1)--(H-4) introduced in  section \ref{sec-classic2}.
% We will refer  to as   "data"  for  any of the constants $\|g\|_{C^2_{\partial \Omega}}, \|h\|_{C^2}, \theta, \lambda, \rho$. 
We recall the notation $\Omega(g) = \{ x\, | \, g(x) >0\}$ and 
$\|g\|_{C^2_{\partial \Omega}} = \max_{\partial \Omega} (|D_{ij} g|  + |D_i
g| + g)$.

\smallskip 

%\begin{thm} \label{theorem-c-2} Assume that $g$ is a classical
%solution of equation  \eqref{eqn-g2} in $\Omega$ 
%with  $0 < p < 2$  and   $h \in C^2(\Omega)$ satisfying \eqref{eqn-lambda}. 
%Assume in addition that $g$ satisfies the assumptions (H-1)--(H-4).
%Define the matrix
%\begin{equation}\label{eqn-matrix20}
%\mathcal M= (\mu_{ij})  = \left (
%\begin{split}
%g\, g_{\nu\nu}+ \theta \, g_\nu^2 &  \quad \sqrt{g} \, g_{\nu\tau}\\
%\sqrt{g} \, g_{\nu\tau}  \,\,\,\, & \quad \,\,\, g_{\tau\tau}
%\end{split}\right )
%\end{equation}
%with $\nu$, $\tau$ denoting the outer normal and tangent
%direction to the level sets of $g$ respectively.
% Then,  there exists  $c=c(\|g\|_{C^2_{\partial \Omega}}, \|h\|_{C^2}, p, \lambda, \rho\, )>0$,   for which  the  bounds
%\begin{equation}\label{eqn-grad0}  c \leq |Dg| \leq c^{-1},
%\end{equation}
%and
%\begin{equation}\label{eqn-matrix30}
% c\, |\xi|^2  \leq  \mu_{ij} \xi_i \, \xi_j  \leq c^{-1}\, |\xi |^2, \qquad \forall \xi \neq 0
% \end{equation}
%hold on $\overline{\Omega (g)}$.
%\end{thm}

We will first establish an upper bound on the
first order derivative $|Dg|$.

%%%%%%%%%%%%%%%%%%%%%%%%%%%%%%%%%%%%%
\begin{lem}\label{lemma-bii1} Under the assumptions of  Theorem \ref{theorem-c-2}, we have
$$\max_{\Omega} |Dg| \leq C(\rho,\theta, \lambda, \|g\|_{C^2_{\partial \Omega}},\|h\|_{C^1} ). $$
\end{lem}
\begin{proof}
We set $M :=r^2 \, |Dg|^2=(x^2+y^2)\, (g_x^2 + g_y^2)$. We will
show  that $M$ attains its maximum at $\partial \Omega$. This
readily implies  the desired bound, since $r^2=x^2+y^2 \geq
\rho^2$ on $\Omega(g)$.  (Notice that we cannot
bound $|Dg|^2$ from above by the maximum principle,  if $h \neq 1$, so we need to
multiply by $r^2$). 

Let $P_0$ be the  maximum point of  $M$ on $\overline {\Omega(g)}$.
Assume first that $P_0 \in \Omega(g)$. We may also assume, by
rotating the coordinates,  that
\begin{equation}\label{eqn-bii1}
g_y=0 \quad \mbox{and} \quad g_x >0 \qquad \mbox{at} \,\, P_0.
\end{equation}
 Also, since $M_x=M_y= 0$ at $P_0$, we have
\begin{equation}\label{eqn-bii2}
g_{xx} =-\frac {x\, g_x}{r^2}, \quad \mbox{and} \quad
g_{xy}=-\frac {y\, g_x}{r^2} \quad \mbox{at} \,\, P_0
\end{equation}
which  combined with \eqref{eqn-bii1} and \eqref{eqn-g2} gives
that
$$
\left (\te  g_x^2 - \frac{xg g_x}{r^2} \right )  g_{yy} -
\frac{y^2 g  g_x^2}{r^4}=h
$$
and hence
\begin{equation}\label{eqn-bii100}
g_{yy} = \frac{r^4\, h + y^2\,g\,  g_x^2}{r^2\, g_x ( \te \, r^2
g_x -xg )}
 \quad \mbox{at} \,\, P_0.
\end{equation}
Let
\begin{equation}\label{eqn-matrix}
A=(a_{ij})=(g\, g_{ij} + \te \, g_ig_j)^{t}
%A= (a_{ij})  = \left (
%\begin{split}
%g g_{yy}+\theta  g_y^2 &  \quad -(g g_{xy}+\theta  g_xg_y) \\
%-(g\, g_{xy}+\theta  g_xg_y) & \quad g g_{xx}+\theta  g_x^2
%\end{split}\right )
\end{equation}
denote the transpose  of the matrix  $G=(G_{ij})=(g\, g_{ij} +
\te \, g_ig_j)$. This is the second order derivative coefficient
matrix of the linearization of equation \eqref{eqn-g2}.
Differentiating equation \eqref{eqn-g2} to eliminate the third order
derivatives on $a_{ij} \, M_{ij}$ and  using  \eqref{eqn-bii1},
\eqref{eqn-bii2} and \eqref{eqn-bii100}, we find, after a  direct
calculation,   that
\begin{equation}\label{eqn-bii123}
a_{ij} \, M_{ij} = \sum_{i=0}^6 \frac{b_i\, M^i}{D}
\end{equation}
with
$$ D= M  r^5 (M\te r -   x g) \quad \mbox{and} \quad b_6= 2 \te (\te x^2 +   y^2)$$
and
$$  |b_i| \leq C(\te, \|h\|_1), \quad  i=1,...,5.$$
Since $r \geq x$, assuming that $M > \te^{-1} \max_{\Omega}  g$ we
conclude that $D >0$ at $P_0$. Since the leading order term in
\eqref{eqn-bii123}, when $M$ is sufficiently large, is $(b_6\,
M^6)/D$  and $b_6 >0$ we conclude that   either $M \leq C(\rho,p,
\lambda, \|g\|_{C^2_{\partial \Omega}},\|h\|_{C^1} )$ at $P_0$ or
$a_{ij} \, M_{ij} >0$. In the latter case $P_0$ cannot be a
maximum point, contradicting our assumption.

Assume next that $P_0 \in \Gamma (g)$ and that $M >0$ at
$P_0$. We may assume again that \eqref{eqn-bii1} holds at $P_0$,
i.e.  $y$ is a tangential direction to $\Gamma (g)$.
Hence, $M_y=0$, $M_x \leq 0$ and $M_{yy} \leq 0$ also hold at
$P_0$. In addition, since $g=0$ at $P_0$,  equation \eqref{eqn-g2}
and \eqref{eqn-bii1} imply that $\te\, g_x^2\, g_{yy} =h $ at
$P_0$. We conclude, after some direct calculations, that
\begin{equation}\label{eqn-bii3}
g=0, \,\, g_{xx} \leq  - \frac{x g_x}{r^2}, \,\, g_{xy}= -
\frac{y g_x}{r^2}, \,\, g_{yy}=\frac {h}{\te g_x^2} \quad
\mbox{at} \,\, P_0
\end{equation}
and
\begin{equation}\label{eqn-bii3'}
M_{yy} = 2\, r^2\,  g_x \, g_{xyy} + \frac{2 (x^2-2y^2)\,
g_x^2}{r^2} + \frac{ 2 r^2 h^2}{\te^2 g_x^4} \leq 0, \qquad
\mbox{at} \,\, P_0.
\end{equation}
On the other hand, differentiating equation \eqref{eqn-g2} with
respect to $x$ and using \eqref{eqn-bii1} and \eqref{eqn-bii3} we
find  that
$$ \te \, g_x^2 \, g_{xyy} - \frac{(1+2\te) y^2 g_x^3}{r^4} - h_x - \frac{(1+2\te) x h}{r^2}=0, \quad \mbox{at} \,\, P_0$$
which implies that $g_{xyy}=-\te^{-1} (1+2\,\te)\, g_{xx} / \,
g_x^3$. Substituting in \eqref{eqn-bii3'} gives that
$$ M_{yy} = \frac{ 2(\te x^2+y^2) M^2}{\te r^4} + \frac{2 (1+2\te)r x h +r^3 h_x}{\te \sqrt{M}} + \frac{2 r^6 h^2}{M^2} \leq 0$$
which is impossible,  if we assume that $M$ is sufficiently large,
depending on the data. This finishes the proof.
\end{proof}

%%%%%%%%%%%%%%%%%%%%%%%%%%%%%%%%%%%
We will next provide a bound from below on  $|Dg|$.

\begin{lem}\label{lemma-bii2}  Under the assumptions of Theorem \ref{theorem-c-2}, we have
$$ |Dg| \geq c(\rho, \te, \lambda, \max_{\partial \Omega} |Dg|,\|h\|_{C^1} ) > 0, \qquad \mbox{on} \,\, \Omega(g). $$
\end{lem}
\begin{proof}
For $q >0$ we set
$$M := (x^2+y^2)^{-q} \, (x\, g_x + y\, g_y)=r^{-2q}\, g_r, \qquad (x,y) \in \Omega(g)$$
with $g_r$ denoting the radial derivative of $g$.

\noindent{\em Claim:} There exists an integer  $q >1$ which
depends only on data (on $\|h\|_{C^1}$ and $\te$) and such that $M
\geq c(\rho, \te, \max_{\partial \Omega} |Dg|,\|h\|_{C^1} ) > 0$
on $\Omega(g)$.
Since, from condition (H-2)  we have $r^2=x^2+y^2 > \rho^2$ for
any $(x,y)\in \Omega(g)$, the claim readily implies  the desired
bound from below on $|Dg|$.

We will next prove the claim by the maximum principle. Let $P_0=(x_0,y_0)$ be an interior  minimum  point of  $M$ in $
\Omega(g)$.  We may  assume, by rotating the coordinates,  that
\begin{equation}\label{eqn-bii4}
y=0 \quad \mbox{and} \quad x > \rho >0 \qquad \mbox{at} \,\, P_0.
\end{equation}
Since  $M_x=0$ and $M_y= 0$ at $P_0$  we have
$$
x\, g_{xx} - (2q-1) \, g_x =0 \quad \mbox{and} \quad x\, g_{xy}+
g_y=0\qquad \mbox{at} \,\, P_0
$$
and hence
\begin{equation}\label{eqn-bii5}
g_{xx} = \frac{(2q-1) \, g_x}{x}  \quad \mbox{and} \quad g_{xy}
-\frac{g_y}x=0 \qquad \mbox{at} \,\, P_0.
\end{equation}
Substituting the above  to equation \eqref{eqn-g2},  using also
\eqref{eqn-bii4}, gives
\begin{equation}\label{eqn-bii555}
g_{yy} = \frac{x^2(1+h) + g_y^2 \, [ g - (2q+1)\te x g_x ]}{ x
g_x \, [(2q-1) g +\te x g_x] }  \qquad \mbox{at} \,\, P_0.
\end{equation}
Let $A=(a_{ij})$ be the matrix defined in $\eqref{eqn-matrix}$.
Differentiating equation \eqref{eqn-g2}  and  \eqref{eqn-bii4} - \eqref{eqn-bii555} we
find, after several direct calculations,   that
$$L:=a_{ij} \, M_{ij} =  \Sigma_{i=0}^4 \frac{b_i \, M^i}{D} $$
with $D=(M \te x^{2q} + (2q-1) g)  >0$ and
$$|b_i| \leq C(\rho, \te, \lambda,  \max_{\partial \Omega} |Dg|, \|h\|_{C^1}) $$
and
$$b_0= - (2q-1)\, x^{-2q-2}\, g \, [ 2(q-2)x^2 (h+1) + 2q\, g g_y^2 -x^3h_x].$$
By choosing  $q >1$ sufficiently large (depending on
$\|h\|_{C^1}$) so that  $$2(q-2) (h+1)  - x\, h_x >0$$ we can make
$b_0 < 0$. We   conclude from the above that $L \leq 0$ unless
$M(P_0) \geq c >0$, for some  constant $c=c(\rho, \theta,\lambda, 
\max_{\partial \Omega} |Dg|, \|h\|_{C^1})$.  This shows that an
interior minimum of $M$  must satisfy $\min M \geq  c(\rho, \theta, \lambda,
\max_{\partial \Omega} |Dg|) >0$.

Assume next that $P_0 \in \Gamma(g)$ is a minimum point for
$M$. We may assume this time that  \eqref{eqn-bii1} holds at
$P_0$. Hence,  \begin{equation}\label{eqn-bii6} M_y=
\frac{-2qxyg_x + r^2 ( y\, g_{yy} + x\, g_{xy}) }{r^{2(q+1)}} =0
\end{equation}
and
\begin{equation}\label{eqn-bii7}
M_x = \frac{[(1-2q)\, x^2+y^2] \, g_x + r^2 (y\, g_{xy} + x
g_{xx})}{r^{2(q+1)}} \geq 0
\end{equation}
and also, by equation \eqref{eqn-g2},  $\te \, g_x^2\, g_{yy} =
h$ at $P_0$. Substituting $g_{yy} =h/(\te\, g_x^2)$ in
\eqref{eqn-bii6} and solving   with respect to $g_{xy}$  gives
\begin{equation}\label{eqn-bii8}
g_{xy} = - \frac{y h}{\te x g_x^2} +\frac{2 p  y g_x}{r^2}.
\end{equation}
Substituting this in \eqref{eqn-bii7} and solving with respect to
$g_{xx}$ gives
\begin{equation}\label{eqn-bii9}
g_{xx} \geq  \frac{y^2 h}{\te \, x^2 g_x^2} + \frac{[(2q-1)x^2-
(2q+1) y^2] g_x}{r^2\, x}.
\end{equation}
Here we have used that $x >0$ at $P_0$. This follows from
assumption  \eqref{eqn-bii1},  (H-2) and the convexity of $\Lambda
(g)$.

We next differentiate equation \eqref{eqn-g2} with respect to $x$
and use that $g=0$, $g_y=0$, $g_{yy} =h/(\te g_x^2)$ and
\eqref{eqn-bii8}  to conclude that
\begin{equation}\label{eqn-bii91}
g_{yyy}=\frac{h_y}{\te g_x^2}, \qquad \mbox{at} \,\, P_0.
\end{equation}
Also, we differentiate equation \eqref{eqn-g2} with respect to
$y$ and use that $g=0$, $g_y=0$, $g_{yy} =h/(\te g_x^2)$,
\eqref{eqn-bii8} - \eqref{eqn-bii91},   to conclude that
\begin{equation}\label{eqn-bii10}
g_{xyy} \leq  \frac{h_x}{\te\, g_x^2} + \frac{ (1+2\te)(1-2q)
h}{\te^2\, x\, g_x^2} + \frac{ 4(1+2\te) q^2 y^2 \, g_x}{\te \,
r^4}.
\end{equation}
We next  differentiate $M$ twice with respect to $y$ and use 
\eqref{eqn-bii8} - \eqref{eqn-bii10}, $g_y=0$, $g=0$  and $g_{yy}
=h/ (\te g_x^2)$. We obtain,  after several direct calculations,  that
$$M_{yy} = \frac{x\, g_{xyy}}{r^{2q}} - \frac{2qx [x^2 + (2q-1)y^2 ] g_x}{r^{2(q+2)}} + \frac{2h+yh_y}{\te r^{2q} g_x^2}$$
at $P_0$. Substituting \eqref{eqn-bii10} and $g_x= M r^{2q}
x^{-1}$ in the above  gives, after several calculations, that
$$M_{yy} \leq b_1 \, M + \frac{b_0}{M^2}$$
with
$$b_1 = \frac{2q [ 2q+\te + \te q) y^2 - \te r^2 ]}{\te r^4}$$
and
$$b_0 = \frac{x^2 \, (  [1+4\te -2q(1+2\te) ] \,  h + \te  y h_y + \te x h_x ) }{\te^2 r^{6q}}.$$
Observe that since $h \geq \lambda >0$ we may choose $q$
sufficiently large, depending on $\|h\|_{C^1}$,  to make $b_0 < -
1$. Since, $M_{yy} \geq 0$ at $P_0$ and   $r \geq \rho$ (by our  assumption (H-2))
we conclude that $M \geq c(\rho, \theta,\lambda, 
\max_{\partial \Omega} |Dg|, \|h\|_{C^1}) >0$, finishing the proof.
\end{proof}

We will next establish sharp  upper bounds on  the second order
derivatives of $g$.
We begin by an upper  bound on the rotationally invariant quantity
$$G : =g_x^2 g_{yy} - 2 g_x  g_y  g_{xy} + g_y^2 g_{xx}= g_\nu^2\, g_{\tau\tau}$$
where  $\nu$, $\tau$ denote   the outer normal and tangential
directions  to the level sets of $g$ respectively.  Since the level
sets of $g$ are convex (because the function $f$ is convex) we
have $G \geq 0$.

\begin{lem}\label{lemma-bii3}  Under the assumptions of Theorem \ref{theorem-c-2},  the quantity
$G=g_\nu^2\, g_{\tau\tau}$ satisfies
$$\max_{\Omega(g)} G \leq  C(\theta, \rho, \lambda, \max_{\partial \Omega } G, \|h\|_{C^2}).$$
\end{lem}
\begin{proof}
Set $M := G + |Dg|^2$. We will estimate $M$ by the maximum
principle. Since $g$ is assumed to be in $C^4( \overline{\Omega (g)})$,  and
hence $g\, g_{ij}=0$ at $\Gamma(g)$, it follows from
equation \eqref{eqn-g2} that $M=\te^{-1}\, h+|Dg|^2$ at $\Gamma(g)$.
Hence, we only need to control $M$ in the interior of
$\Omega(g)$. 
 Assuming that the maximum of $M$ is attained at an interior point $P_0 \in \Omega(g)$, we will show that
\begin{equation}\label{eqn-in1}
 a_{ij} \, M_{ij}  = \frac 1{M^2\, (1+g_x^2)} \, \sum_{i=0}^4 A_i \, M^i , \qquad \mbox{at} \,\, P_0
\end{equation}
with $A=(a_{ij})$ given by \eqref{eqn-matrix}, $A_4 \geq c (\theta, \rho, \lambda, \max_{\partial \Omega } G, \|h\|_{C^2}) >0$
and $|A_i| \leq C(\theta, \rho, \lambda, \max_{\partial \Omega } G, \|h\|_{C^2})$, for $i=0,...3$. Since $ a_{ij} \, M_{ij} \leq
0$ at a maximum point, this will imply the inequality $M \leq
 C(\theta, \rho, \lambda, \max_{\partial \Omega } G, \|h\|_{C^2})$, showing  that $\max_{\Omega(g)}  M \leq
C(\max_{\partial \Omega} M, \theta, \rho, \lambda, \max_{\partial \Omega } G, \|h\|_{C^2})$, as desired.

To prove \eqref{eqn-in1}, we begin by noticing that since $M$ is
rotationally invariant we may assume that \eqref{eqn-bii1} holds
at $P_0$, i.e.,  $g_x >0$, $g_y=0$ and $M=g_x^2\, (g_{yy}+1)$  at
$P_0$. Also, by a standard change of variables (see in the proof of Proposition 4.1 in  \cite{S}),
we may also assume that $g_{xy} =0$ at $P_0$. Using
\eqref{eqn-bii1} we compute that
$$ M_x=  g_x^2\, g_{xyy}+ 2\, g_x g_{xx}(1+g_{yy}) =0,  \quad M_y=g_x^2\, g_{yyy}=0 \quad \mbox{at}\,\, P_0
$$
implying that
\begin{equation}\label{eqn-bii13}
g_{xyy} = - \frac{2\, g_{xx} \, (1+g_{yy})}{g_x}, \qquad
g_{yyy}=0, \qquad \mbox{at}\,\, P_0.
\end{equation}
Differentiating equation \eqref{eqn-g2} in  $y$, using
\eqref{eqn-bii13} and solving with respect to $g_{xxy}$ we obtain
\begin{equation}\label{eqn-bii14}
g_{xxy} =  \frac{h_y}{g\, g_{yy}} \qquad \mbox{at}\,\, P_0
\end{equation}
since $g_y=g_{xy}=0$ at $P_0$. Also,  differentiating equation
\eqref{eqn-g2} in  $x$, using \eqref{eqn-bii13}-\eqref{eqn-bii14}
and solving with respect to $g_{xxx}$ we obtain
\begin{equation}\label{eqn-bii15}
g_{xxx} =  \frac{h_xg_x+ g_{xx}\, [(2\te-g_{yy})g_x^2+2g\,
(1+g_{yy})g_{xx}]}{g\, g_x g_{yy} } \quad \mbox{at}\,\, P_0.
\end{equation}
We next differentiate equation \eqref{eqn-g2} twice in $y$,
multiply it by $g_x^2$ and subtract it  from $a_{ij} \, M_{ij}$ to
eliminate fourth order derivatives, while use
\eqref{eqn-bii13}-\eqref{eqn-bii15} to eliminate third order
derivatives. After several direct calculations, using also that
$g_y=0=g_{xy}=0$ at $P_0$,  we obtain that
$$a_{ij} \, M_{ij} = \frac 1{M^2\, (1+g_x^2)} \, \sum_{i=0}^4 A_i \, M^i$$
with
$$A_4 = 3\te (1+4\te) g_x^4$$
and
$$ |A_i| \leq C(\te, \rho,\lambda,  \|h\|_{C^2}, \|g\|_{C^1}), \qquad i=0,\cdots,3.$$
By the previous two Propositions, $0 < c \leq g_x \leq C <
\infty$. Hence, $A_4 >0$, while $A_i$, $i=0, \cdots, 3$ bounded. This shows that
at an interior maximum point,  $M \leq  C(\theta, \rho, \lambda, \max_{\partial \Omega } G, \|h\|_{C^2})$, hence 
finishing  the proof of the Lemma.
\end{proof}

We will now bound
\begin{equation}\label{eqn-bii16}
Q :=\max_{\gamma} \, (g \, D_{\gamma\gamma} g + \te\,
|D_{\gamma}g|^2), \qquad \te=\frac{1+p}{2-p}
\end{equation}
from above, where the maximum in \eqref{eqn-bii16} is taken over
all  directions $\gamma$.  Note, that in terms of the function
$f$, we have 
$$Q = \max_{\gamma} \, (q^{1/3} \, f^{\frac{1-2p}{3}}\, f_{\gamma\gamma}), \qquad q=\frac{3}{2-p}.$$
In particular, since $f$ is convex, $Q \geq 0$.

\begin{lem}\label{lemma-bii4}
 Under the assumptions of Theorem \ref{theorem-c-2}, we have
\begin{equation}\label{eqn-bii17}
\max_{\Omega(g)} Q  \leq C( \te, \rho, \lambda, \|g\|_{C^2_{\partial
\Omega}}, \|h\|_{C^2}).
\end{equation}
\end{lem}
\begin{proof}
We begin by observing that since  $g \in C^4(\overline{ \Omega (g)} )$,
by Lemma \ref{lemma-bii1}, the bound $Q=\te \, |D g|^2 \leq C(\theta, \rho, \lambda, \|g\|_{C^2_{\partial \Omega}}
, \|h\|_{C^1})$ holds.

Assume next that the maximum of $Q$  is attained at an interior
point $P_0 \in \Omega(g)$ and at a direction $\gamma$, so that
$$Q(P_0) = g \, D_{\gamma\gamma} g + \te\, |D_{\gamma}g|^2.$$
Let $\nu$, $\tau$ denote the outward normal and tangential directions to the level sets of
$g$ respectively.

\noindent{\em Claim.} Either $Q(P_0) \leq C( \te, \rho, \lambda, \|g\|_{C^2_{\partial
\Omega}}, \|h\|_{C^2})$ or $ \gamma = \nu.$

To prove the claim,  we begin by expressing  the maximum
direction $\gamma$ as $\gamma = \lambda_1 \, \nu + \la_2 \,
\tau$, with $\la_1^2+\la_2^2=1$ so  that
\begin{equation}\label{eqn-bii18}
Q(P_0) = g\, [\la_1^2 \, g_{\nu\nu} + 2\, \la_1\la_2 \,
g_{\nu\tau} + \la_1^2 \, g_{\tau\tau}] + \la_1^2 \, g_\nu^2.
\end{equation}
Next, we use the equation \eqref{eqn-g2} expressed in the form
$$(g\, g_{\nu\nu} + \te\, g_\nu^2)\,  g_{\tau\tau} =  g\, g_{\nu \tau}^2  + h$$
and the bounds in Lemmas   \ref{lemma-bii1} -
\ref{lemma-bii3}  to first conclude the bound  $$Q(P_0) \leq C( \theta,
\rho, \lambda, \|g\|_{C^2_{\partial \Omega}},\|h\|_{C^2}\,  )$$ unless $g
g_{\nu\nu}$ is sufficiently large at $P_0$. If in particular  $g
g_{\nu\nu} > \te g_\nu^2$ at $P_0$, we  then conclude from
\eqref{eqn-g2} and Lemmas  \ref{lemma-bii2} and
\ref{lemma-bii3} that
$$g\, g_{\nu\tau}^2  \leq 2\, g\, g_{\nu\nu} \, g_{\tau \, \tau}
 \leq C( \te, \rho, \lambda, \|g\|_{C^2_{\partial
\Omega}}, \|h\|_{C^2}) \, g\, {g_{\nu\nu}}$$
showing the bound $g_{\nu\tau} \leq C( \te, \rho, \lambda, \|g\|_{C^2_{\partial
\Omega}}, \|h\|_{C^2})\, \sqrt{g_{\nu\nu}}.$ Using once more
the bound $g_{\tau\tau} \leq C( \te, \rho, \lambda, \|g\|_{C^2_{\partial
\Omega}}, \|h\|_{C^2}),
\theta \,)$, we readily deduce from \eqref{eqn-bii18} that
$Q(P_0)$ becomes maximum when $\lambda_2=0$, provided  it  is
sufficiently large, depending   only on $\|g\|_{C^2_{\partial
\Omega}}, \theta, \rho$. This proves the Claim.

If $Q(P_0) \leq C( \te, \rho, \lambda, \|g\|_{C^2_{\partial
\Omega}}, \|h\|_{C^2})$
then the proof of the Proposition is complete. Otherwise,  from the previous claim we may
assume  that $Q(P_0) = g\, g_{\nu\nu} + \te\, g_\nu^2$ and also,
since $\nu$ is the maximum direction, that $g\, g_{\nu\tau} + \te
\, g_{\nu} g_{\tau} = 0$ at $P_0$, implying that $g_{\nu\tau} =0$
at $P_0$, since $g >0$ and $g_{\tau}=0$ at $P_0$. Also, by
rotating the coordinates, we may assume that \eqref{eqn-bii1}
holds at $P_0$, i.e., the direction of the vector $\nu$ is that
of the $x$-axis.

We will show that
\begin{equation}\label{eqn-in2}
0 \geq a_{ij} \, Q_{ij}  = g\, [ (1+h)\, Q + C( \te, \rho, \lambda, \|g\|_{C^2_{\partial
\Omega}}, \|h\|_{C^2})] \qquad \mbox{at} \,\, P_0
\end{equation}
with $A=(a_{ij})$ given by \eqref{eqn-matrix}.  Since $h >0$, this implies
the bound  $Q \leq C(\te, \|h\|_{C^2}, \|g\|_{C^1})$  at $P_0$,  which
combined with Lemma \ref{lemma-bii2} implies  the  desired estimate. 

To prove \eqref{eqn-in2} let us first summarize that
\begin{equation}\label{eqn-bii19}
g_x >0, \quad g_y=0, \quad g_{xy}=0 \qquad \mbox{at} \,\, P_0.
\end{equation}
Also, since $Q= g\, g_{xx}+\te g_x^2$ at $P_0$, equation
\eqref{eqn-g2} together with conditions \eqref{eqn-bii19} imply
that
\begin{equation}\label{eqn-bii20}
g_{yy} = h \, Q^{-1} \qquad \mbox{at} \,\, P_0.
\end{equation}
We next differentiate $Q$ is $x$ and $y$ and use \eqref{eqn-bii19}
to deduce the equalities
$$ Q_x=g\, g_{xxx} +  (1+2\, \te)\, g_x \, g_{xx},   \qquad G_y=g\,  g_{xxy}=0 \quad \mbox{at}\,\, P_0
$$
from which we conclude that
\begin{equation}\label{eqn-bii21}
g_{xxx} = - \frac{(1+2\, \te)\, g_x \, g_{xx}}{g},   \qquad
g_{xxy}=0, \qquad \mbox{at}\,\, P_0.
\end{equation}
Also, differentiating equation \eqref{eqn-g2} in  $x$, using
\eqref{eqn-bii19} and  \eqref{eqn-bii21} gives
 \begin{equation}\label{eqn-bii22}
g_{xyy} = \frac{h_x}{Q} \qquad \mbox{at}\,\, P_0.
\end{equation}
We next differentiate twice the equation \eqref{eqn-g2} in $x$ to
eliminate the fourth order derivatives from $a_{ij}\, Q_{ij}$ and
use \eqref{eqn-bii21}- \eqref{eqn-bii22} to eliminate third order
derivatives and also \eqref{eqn-bii19}-\eqref{eqn-bii20} to
finally conclude, after several calculations, that
$$ a_{ij} \, Q_{ij} = g\, [ (1+h)\, G + C( \te, \rho, \lambda, \|g\|_{C^2_{\partial
\Omega}}, \|h\|_{C^2})$$
 by the previous
Lemmas  and our assumptions. This finishes the proof of the
Lemma.
\end{proof}

We are now going to combine the estimates in Lemmas
\ref{lemma-bii1}-\ref{lemma-bii4} to give the proof of Theorem
\ref{theorem-c-2}.

\medskip

\noindent{\em Proof of Theorem \ref{theorem-c-2}.} We begin by
expressing \eqref{eqn-g2} in the form $$ \det \mathcal M = (\,
g\, g_{\nu\nu} + \te g_\nu^2 \, ) \, g_{\tau\tau} - g\,
g_{\nu\tau}^2=h.$$ Hence,  it is enough to establish the bounds
$$c \leq  g\, g_{\nu\nu}+ \theta \, g_\nu^2 \leq c^{-1} \quad \mbox{and}
\quad c \leq  g_{\tau\tau} \leq c^{-1}$$ for a constant
$c=c(\theta, \rho, \lambda, \|g\|_{C^2_{\partial \Omega}},
\|h\|_{C^2}\, )>0$.

The bounds from  above  readily follow from Lemmas \ref{lemma-bii3} and 
\ref{lemma-bii4} combined with Lemma \ref{lemma-bii2}. The bounds from
below follow  from $(\, g\, g_{\nu\nu} + \te g_\nu^2 \, ) \, g_{\tau\tau} = h+ g\,
 g_{\nu\tau}^2 \geq \lambda >0$ (from our assumption on $h$)
and the corresponding  bounds from above. \qed

We next  re-state  Theorem \ref{theorem-c-2} in terms  of
the solution $f$ of \eqref{eqn-f2}.

\begin{corollary}\label{cor-bii} Assume that $f$ is a non-negative weakly convex classical   solution $f$ of the boundary value problem  \eqref{eqn-f2} in $\Omega$, with  $0 < p < 2$,  which satisfies assumptions
(H-1)--(H-4). Define the matrix
\begin{equation}\label{eqn-matrix2}
\mathcal M= (\mu_{ij})  = \left (
\begin{split}
q^{\frac 13}\, f^{\frac{1-2p}3}\, f_{\nu\nu} & \qquad f^{-\frac p2}  \, f_{\nu\tau}\\
f^{-\frac p2}  \, f_{\nu\tau}  \,\,\,\, & \quad \,\,\, q^{-\frac
13} f^{-\frac{1+p}3}\, f_{\tau\tau}
\end{split}\right )
\end{equation}
with $\nu$, $\tau$ denoting the outer normal and tangent
direction to the level sets of $f$ respectively and $q=3/(2-p)$.
 Then,  there exist a constant $c=c(\, \theta, \rho, \lambda,
 \|f^q\|_{C^2_{\partial \Omega}}, \|h\|_{C^2}\,  )>0$   for which
\begin{equation}\label{eqn-matrix33}
 c\, |\xi|^2  \leq  \mu_{ij} \xi_i \, \xi_j  \leq c^{-1}\, |\xi |^2,
 \qquad \forall \xi \neq 0.
 \end{equation}
\end{corollary}

We will finish this section with the following lower bound on $\sqrt{g}\, \det D^2 g$,
which will be used in the next section. 

\begin{prop}\label{proposition-detg}
Under the assumptions of Theorem \ref{theorem-c-2}, there exists a constant
$C=C(\|g\|_{C^2_{\partial \Omega}}, \|h\|_{C^2}, \theta, \lambda, \rho\, )>0$,   for which  the  bound
\begin{equation}\label{eqn-detg} \sqrt{g} \,  \det D^2 g  \geq  - C
\end{equation}
holds on $\Omega(g)$.
\end{prop}

\begin{proof}
Set $Z := (x^2+y^2) \, \sqrt{g} \, \det D^2 g$.   We will use the maximum principle to establish the bound $Z \geq -C$, which readily implies \eqref{eqn-detg}, since  $x^2+y^2  \geq \rho^2$ on $\Omega(g)$.

Clearly, $Z \geq - C$ on $\partial \Omega(g)$. Assume  that the maximum of $Z$ is attained at an interior point $P_0 \in \Omega(g)$.
Since $Z$ is rotationally invariant, we may assume, without loss of generality,  that
\begin{equation}\label{eqn-detg1}
g_x >0,  \quad g_y=0  \quad{and} \quad  g_{xy} =0, \quad \mbox{at} \,\,  P_0.
\end{equation}
Differentiating equation once and twice in $x,y$ and using that $Z_x=Z_y=0$ at $P_0$,  we find,  after several direct calculations,  that
at the minimum point $P_0$ where \eqref{eqn-detg2} holds, we have 
\begin{equation}\label{eqn-detg2}
0 \leq  a_{ij} \, Z_{ij}  = \frac 1{4\, r^4\, g\,g_1^2\,g_{22}} \, \sum_{i=1}^3 A_i \, Z^i , \qquad \mbox{at} \,\, P_0
\end{equation}
with $A=(a_{ij})$ given by \eqref{eqn-matrix}, and
$$A_1 =  13 \, (x^2+y^2)^2 \, g_1^4 \, g_{22} + C\, \sqrt{g}$$
and
$$A_2 =
- \sqrt{g} \, (117 \, (x^2+y^2)^{3/2}\,  g_1^2\, g_{22} +  C\, g)  \quad \mbox{and} \quad A_3= -  \frac{4g^2 x_2^2}{(x^2+y^2)^2 g_{22}}.$$
The constants $C=C(g_1,g_{22},x,y)$ depend only on $g_1$, $g_{22}$, $x,y$ and hence
they are bounded, by Theorem \ref{theorem-c-2}.

We will  show that $a_{ij} \, Z_{ij} < 0$ at $P_0$ provided that $Z <0$ is sufficiently
large in absolute value and $P_0$ is sufficiently close to the free-boundary $\Gamma(g)$, establishing a contradiction to $a_{ij} \, Z_{ij} \geq 0$ at the minimum point $P_0$
of $Z$. 

It is clear from the estimates in Theorem \ref{theorem-c-2} that $A_1 \, Z < 0$ and $A_2 \, Z^2 <0$, provided $P_0$ is sufficiently close to the free-boundary $\partial \Omega$, i.e. $g$ is sufficiently close to zero. The term $A_3\, Z^3$ is nonnegative, however
we observe that
$$A_3 \, Z^3 = -    \frac{4 x_2^2\,  g \,  (\sqrt g \, Z)^2}{(x^2+y^2)^2 g_{22}} \, Z = C g  Z $$
with $C$ bounded, since $\sqrt g \, Z$ is bounded by the estimates in Theorem \ref{theorem-c-2}. Hence, $\sum_{i=1}^3 A_i \, Z^i \leq A_1+A_3 <0$ at $P_0$, provided that $Z <0$ is sufficiently large and $P_0$ is sufficiently close to the free-boundary $\Gamma(g)$,
which concludes the proof of the Proposition.
\end{proof}

\section{$C^{2,\alpha}_s$-Regularity}
\label{sec-c2a}

We will assume throughout this section that  $g \in C^4(\overline{\Omega(g)})$ is a classical
solution of the boundary value problem \eqref{eqn-g2} in $\Omega$,
with  $0 < p < 2$ and  $h \in C^2(\Omega)$ satisfying  \eqref{eqn-lambda}. 
In  addition, we assume  that $g$ satisfies the assumptions (H-1)--(H-4).
Our goal is to establish a uniform estimate on the norm $\| g \|_{C^{2,\alpha}_s (\overline{\Omega(g)})}$, as defined in section 2.2,  by combining the a-priori estimates in Theorem \ref{theorem-c-2} with  the H\"older Regularity result Theorem [DL]. 
We will obtain estimates which  depend only
on the data  $\|g\|_{C^2_{\partial \Omega}}, \|h\|_{C^2}, \theta, \lambda, \rho$.

Since the regularity theorem [DL]  concerns with solutions on a
fixed domain,  we will first  perform  a  change of coordinates,
near the interface, which transforms the free-boundary problem
\eqref{eqn-g2} to a nonlinear degenerate  problem with fixed-boundary. The
same  coordinate change was used in \cite{DH2}. We refer the
reader to that paper for the detailed computations.

Let
$P_0=(x_0,y_0) \in \Gamma(g)$ be  a free-boundary  point. We may  assume,
by rotating the coordinates, that at the point $P_0$, \begin{equation}\label{nvec2}
n_0:=   \frac{\bf P_0}{|{\bf P_0}|}  = {\bf e_1}.
\end{equation}
Then, by Theorem \ref{theorem-c-2},  $g_x(P) >0$, for all points $P=(x,y)$  sufficiently close to $P_0$.
Hence, we can solve  around the point $P_0$,  the equation
$z=g(x,y)$ with respect to $x$, yielding to a map
$$x=q(z,y)$$
defined for all $(z,y)$ sufficiently close to $Q_0=(0,y_0)$.
Using  the identities
$$ g_x  =  \frac 1{q_z}, \quad g_y = - \frac{q_y}{q_z},  \quad  g_{xx} = - \frac 1{q_z^3}\,  q_{zz} $$
and
$$ g_{xy}= -  \frac 1{q_z}\,
\left ( -\frac {q_y}{ q_z^2}\,  q_{zz} +  \frac 1{q_z}q_{zy} \right ),
\quad
g_{yy} =  - \frac 1{q_z}\,
\left ( \frac {q_y^2}{q_z^2} \,q_{zz} -
2 \frac {q_y}{ q_z}\,  q_{zy} +  q_{yy} \right )$$
which yield to
$$g_{xx} g_{yy} - g_{xy}^2 = \frac 1{q_z^4}
\left ( q_{zz} q_{yy} - q_{zy}^2 \right ) $$
and
$$ g_y^2g_{xx} - 2g_xg_yg_{xy}+g_x^2g_{yy}
= -\frac 1{q_z^3} \, q_{yy}$$
we find that
$q$ satisfies the equation
\begin{equation}\label{eqn-q}
\frac { - z\, {\det} D^2  q  +  \theta \, q_z \, q_{yy} }
{q_z^4} = - H
\end{equation}
with
\begin{equation}\label{eqn-H}
H(z,y)=h(x,y), \qquad x=q(z,y).
\end{equation}
In addition, $q$ is a concave  function, since $g$ is convex.

Consider the non-linear operator
$$Lq := \frac {- z\, {\det} D^2  q  +  \theta \, q_z \, q_{yy} }
{q_z^4}.$$
The linearization $\tilde L$ of $L$  around a  point  $q$ has the form
\begin{equation}\label{eqn-tilq}
\begin{split}
L_q(\tilde q) = &\frac {- z \, q_{yy} \tilde q_{zz} +  2 z\,  q_{zy} \, \tilde
q_{zy} + (  \theta \, q_z  - z\, q_{zz} )\, \tilde q_{yy}}
{q_z^4}  \\ &\qquad \qquad \qquad \qquad + \frac {4\, z\, {\det} D^2  q  - 3 \,  \theta \, q_z \, q_{yy} }
{q_z^5}  \, \tilde q_z .
\end{split}
\end{equation}
Let us denote by  $\mathcal B_\eta$ the box
\begin{equation}\label{eqn-box}
\mathcal B_\eta =
\{ \, 0\leq z \leq \eta^2
, \,\, |y-y_0| \leq \eta \,\, \}
\end{equation}
around  $Q_0=(0,y_0)$ and by  $C^\alpha_s({\mathcal B}_\eta)$, $C^{2,\alpha}_s({\mathcal B}_\eta)$ the spaces defined in section 2.2. 
Our goal in this section is to establish the following result:

\begin{thm}\label{theorem-c2a}
Assume that  $g \in C^4(\overline{\Omega(g)})$ is a non-negative classical  solution  of the boundary value problem \eqref{eqn-g2} on $\Omega$, with  $0 < p < 2$  and $h \in C^2(\Omega)$  satisfying
condition \eqref{eqn-lambda}. In addition,  assume that $g$  satisfies the assumptions 
(H-1)--(H-4). 
Then, there exist  constants $0 < \alpha < 1$,
$C<\infty$ and $\eta >0$, depending only on the data $\|g\|_{C^2_{\partial \Omega}}, \|h\|_{C^2}, \theta, \lambda, \rho$,  such that  for any free-boundary point
$P_0=(x_0,y_0)$, satisfying condition
\eqref{nvec2},  the function $x=q(z,y)$ satisfies
the estimate
$$\|q\|_{C^{2+\alpha}_s(\B_{\eta})} \leq C$$
on  $\B_\eta= \{ \, 0\leq z \leq \eta^2
, \,\, |y-y_0| \leq \eta \,\, \}$.
\end{thm}

Consider the matrix
\begin{equation}\label{eqn-matrix50}
\mathcal A= (\alpha_{ij}) := q_z^{-4} \left (
\begin{split}
- q_{yy}  & \quad \sqrt z \, q_{zy}\\
\sqrt z \, q_{zy} & \quad   \theta \, q_z  -z\, q_{zz}
\end{split}\right )
\end{equation}
and the coefficient
\begin{equation}\label{coeff}
b :=  \frac {4 \, z\, {\det} D^2  q  -  3\,   \theta \, q_z \, q_{yy} }
{q_z^5}.
\end{equation}

A direct consequence of Theorem \ref{theorem-c-2} is the following
a-priori bounds on $\mathcal A$ and $b$.

\begin{lem}
\label{lemma-boundsA} Under the assumptions of Theorem
\ref{theorem-c2a},  there exist a positive constants  $c=c(\|g\|_{C^2_{\partial
D}}, \|h\|_{C^2}, \theta, \lambda, \rho\, )$ and $\eta_0$,   for
which  the  bounds
\begin{equation}\label{eqn-matrix5}
0 <  c\, |\xi|^2  \leq  \al_{ij} \, \xi_i \, \xi_j  \leq c^{-1}\, |\xi |^2, \qquad \forall \xi \neq 0
 \end{equation}
 and
\begin{equation}\label{eqn-coeffb} 0 <  c \leq  b  \leq c^{-1}
\end{equation}
hold   on the box ${\mathcal B}_\eta$, provided $\eta \leq \eta_0$.
\end{lem}

\begin{proof}
By direct calculation
\begin{equation}\label{det}
\det \mathcal A = \frac {z\, {\det} D^2  q  -  \theta \, q_z \, q_{yy} }
{q_z^4} = h
\end{equation}
and
\begin{equation}\label{eqn-trace}
\text {tr} \, \mathcal A = \frac {1}{g_x^3} \, \left [
 (g_y^2 \, g_{xx} - 2 g_x g_y g_{xy} + g_x^2 \, g_{yy}) + (g g_{xx} + \theta \, g_x^2)
\right ].
\end{equation}
By \eqref{eqn-lambda}, $\lambda < \det \mathcal A < \lambda^{-1}$. The bound
$ c < \text {tr} \, \mathcal A < c^{-1}$  follows from Theorem \ref{theorem-c-2} and \eqref{eqn-trace}. These two bounds yield to \eqref{eqn-matrix5}.

Next,  we observe that
$$b =   \frac {4 z\, {\det} D^2  q  -  3\, \theta \, q_z \, q_{yy} }
{q_z^5} = g_x\, ( 3 \, h + g \, \det D^2 g).$$
Theorem   \ref{theorem-c-2} shows   that  $b \leq c^{-1}$ on ${\mathcal B}_\eta$. The bound from below $b \geq c >0$ on ${\mathcal B}_\eta$, with $\eta$ sufficiently small,  readily  follows from   \eqref{eqn-lambda} and \eqref{eqn-detg}. \end{proof}

We are now in position to show the uniform H\"older  bounds of the
first order derivatives $h_y$ and $h_z$ of $h$ on $\mathcal B_\eta$.

\begin{lem}\label{lemma-qzqy}
Under the assumptions of Theorem \ref{theorem-c2a}, there exists a number
$\alpha \in (0,1)$, and  positive constants $\eta$ and
$C$, depending only on the data  $\|g\|_{C^2_{\partial \Omega}},
\|h\|_{C^2}, \theta, \lambda, \rho\, $, such that
$$
\| q_z\|_{C^\alpha_s(\mathcal B_{\frac \eta2})} \leq C \qquad
\mbox{and}
\qquad
\| q_y\|_{C^\alpha_s(\mathcal B_{\frac \eta2})} \leq C.$$\end{lem}

\begin{proof}
We will first establish the bound for $\tilde q = q_y$.
Differentiating equation \eqref{eqn-q} with
respect to $y$ we find that  $\tilde q= q_y$ satisfies the  equation
$ L_q(\tilde q) = \tilde H$ with $\tilde H = - \partial_y H$, with $L_q$ given by \eqref{eqn-tilq}.
Since  $\partial_y H = h_y + h_x \, q_y$,  using  the notation
$$H_y(z,y) = h_y(x,y) \quad {and} \quad H_z(z,y) = h_x(x,y), \qquad x=q(x,y)$$ we conclude that $\tilde q$ satisfies the equation
\begin{equation}\label{eqn-qy}
z\, \al_{11} \, \tilde q_{zz} + 2 \sqrt z \al_{12} \, \tilde q_{zy} + \al_{22}\, \tilde q_{yy} + b \, \tilde q_z  + c\, \tilde q = - H_y
\end{equation}
with $\alpha_{ij}$ and $b$  given by \eqref{eqn-matrix50} and  \eqref{coeff}  respectively
and $c=h_x(x,y)=H_z(z,y)$. In addition,  Lemma \ref{lemma-boundsA} and our conditions on the function $h$,   imply that equation
\eqref{eqn-qy} satisfies all the assumptions of our
 $C^\alpha$-regularity result, Theorem [DL].  Hence, there exists
a number $\alpha$ in $0 < \alpha <1$, such that the H\"older norm
$\|\tilde q\|_{{C^{\alpha}_s({\mathcal B}_{\frac \eta
2})}}$ is bounded in terms of $\|\tilde h\|_{C^0({\mathcal B}_{\eta })}$
and $\|H_y\|_{C^0({\mathcal B}_{\eta })}$.
Since $\|\tilde q\|_{C^0({\mathcal B}_{\eta })}$  is
uniformly bounded, the bound $\| q_y\|_{C^\alpha_s(\mathcal B_{\frac \eta2})} \leq C$ readily follows from our assumptions on the function $h$.

We will now establish the $C^\alpha_s$ bound for $\tilde q= q_z$.
Differentiating equation \eqref{eqn-q} with
respect to $z$ we find that  $\tilde q= q_z$ satisfies the equation
\begin{equation}\label{eqn-tilqz}
\begin{split}
&\frac { z \, q_{yy} \tilde q_{zz} -  2 z\,  q_{zy} \, \tilde
q_{zy} + ( z\, q_{zz} - \theta \, q_z )\, \tilde q_{yy}}
{q_z^4}  \\ &  \qquad \qquad \frac {4\, z\, {\det} D^2  q  - (3 \theta +1) \, q_z \, q_{yy} }
{q_z^5}  \, \tilde q_z + \frac{q_{zy}^2}{q_z^5} = H_1
\end{split}
\end{equation}
with $H_1 = \partial_z H = h_x\, q_z = H_z\, q_z$.
We wish to apply  the regularity Theorem [DL] shown in  \cite{DL2} to control the
$C^\alpha_s$ norm of $\tilde q=q_z$. However,
our a-priori bounds in Theorem \ref{theorem-c-2}   do not imply that the term  $
{q_{zy}^2}/{q_z^5}$ is bounded, since the bounds  \eqref{eqn-matrix5}
only  control  $\sqrt{z}\, h_{zy}$.

To control the $C^\alpha_s$ norm of $h_z$, we will apply Theorems
3.6 and Theorem 3.7 in \cite{DL1} on certain  super-solutions and
sub-solutions  of equation \eqref{eqn-tilqz}.

We begin by noticing that since the  term  ${q_{zy}^2}/{q_z^5}$
is nonnegative, \eqref{eqn-tilqz} implies that $\tilde q =q_z$
is a super-solution of equation
\begin{equation}\label{eqn-tilqz2}
\begin{split}
&\frac { z \, q_{yy} \tilde q_{zz} -  2 z\,  q_{zy} \, \tilde
q_{zy} + ( z\, q_{zz} - \theta \, q_z )\, \tilde q_{yy}}
{q_z^4}  \\ &  \qquad \qquad \frac {4\, z\, {\det} D^2  q  - (3 \theta +1) \, q_z \, q_{yy} }
{q_z^5}  \, \tilde q_z \leq   H_1.
\end{split}
\end{equation}
Let us denote  by  $( a_{ij})$ the matrix in \eqref{eqn-matrix50} and  by
\begin{equation}\label{eqn-b1}
b_1 := \frac {4\, z\, {\det} D^2  q  - (3 \theta +1) \, q_z \, q_{yy} }
{q_z^5}
\end{equation}
and set
\begin{equation}
L_1 (\tilde q) :=   z  a_{11}  \tilde
q_{zz} +  2
\sqrt z\, a_{12} \, \tilde q_{zy} +  a_{22} \,  \tilde q_{yy}
+  b_1\,  \tilde q_z.
\label{eq-linear}
\end{equation}

A similar argument  to that  used in the proof of  \eqref{eqn-coeffb} shows that  $b_1$  satisfies  the bounds
\begin{equation}\label{eqb-boundsb1}
 c < b_1< c^{-1}, \qquad \mbox{on}\,\, \mathcal B_\eta
 \end{equation}
 with $c=c(\|g\|_{C^2_{\partial \Omega}}, \|h\|_{C^2}, \theta, \lambda, \rho\, ) >0$.

Following very similar  computations to those in  the proof of Lemma 5.9 in in \cite{DL1}, we conclude:

\begin{itemize}
\item $\tilde q = q_z$ is a super-solution of equation
$$\label{eqn-G1}
 L_1(\tilde q) \leq  \tilde H_1, \qquad \mbox{on} \,\, \mathcal B_\eta
$$
with $\tilde H_1 = H_z  \, \tilde q$.

\item  There exists a number $\beta  >1$, depending only
on the a priori bounds, for which if $(h_z - m) > 0$ on ${\mathcal
B}_\eta$, for some positive constant $m$, then $\tilde q_2:=(h_z
-m)^\beta$  is a sub-solution of the equation
$$\label{eqn-H2}
L_1 (\tilde q_2) \geq H_2.
$$
\item  There exists a number $\beta >1$, depending only
on the a priori bounds, so that   $\tilde q_3:=h_z^\beta$  is a
sub-solution of the equation
$$\label{eqn-H3}
 L_1 (\tilde q_3) \geq  H_3.
$$

\item  There exists a number $\beta >1$, depending only on the a  priori bounds,
so that  for any constant $M$,  $\tilde q_4:=(M^\beta
-h_z^\beta)$  is a super-solution of the equation
$$\label{eqn-G4}
L_1 (\tilde q_4)  \leq  H_4.
$$
\end{itemize}
It can be shown,   as in the proof of Lemma 5.9 in \cite{DL1},  that the functions $H_i$, $i=1,..,4$ satisfy the bounds
$$ \|H_i\|_{L^\infty(\mathcal B_\eta)} \leq C(\|g\|_{C^2_{\partial \Omega}}, \|h\|_{C^2}, \theta, \lambda, \rho\, ).$$

The H\"older regularity of the function $\tilde h=h_z$ on ${\mathcal B_\eta}$ follows by 
combining  the above  with the Harnack estimate,
Theorem 3.6,  and the local maximum principle, Theorem 3.7  in \cite{DL2},
along the lines of the proof of Lemma 5.9 in \cite{DL1}. This yields to  the bound  $\| q_z\|_{C^\alpha_s(\mathcal B_{\frac \eta2})} \leq C(\|g\|_{C^2_{\partial \Omega}}, \|h\|_{C^2}, \theta, \lambda, \rho\, )$.
\end{proof}

We will next combine  Lemmas \ref{lemma-boundsA} and  \ref{lemma-qzqy}
with the classical  regularity results for strictly
elliptic  linear and fully nonlinear equations,  to obtain the
$C^{2,\alpha}_s$ regularity of the solution $q$ on the box
$\mathcal B_\eta$ defined by \eqref{eqn-box} around  the boundary point $Q_0=(0,y_0,t_0)$, where Lemma \ref{lemma-boundsA}  holds.

Let $Q^r=(r^2,y_r)$ be a point in $\B_\eta$, where the index $r$
indicates that the $z$ coordinate of $Q_r$ is of distance $r^2$ from the
boundary $z=0$. For $0 < \mu < 1$,   denote by $D_\mu$ the disk
$D_\mu  = \{ \,  z^2 + y^2  \leq \mu^2\,  \}$.
Define the dilation  $q^r$ of $q$ on $D_\mu$, namely the function 
$$q^r(z,y) := \frac{q(r^2 + r^2 z, y_r + r\, y)}{r^2}.$$
A direct computation shows that
the function $q^r$ satisfies the  equation
\begin{equation}\label{eqn-qr}
\frac { - \tilde z\, {\det} D^2  q^r  +  \theta \, q_z^r \, q_{yy}^r}
{(q_z^r)^4} = - H^r
\end{equation}
with $\tilde z = 1+z$ and $H_r(z,y)=H(r^2 + r^2 z, y_r + r\, y)$.

When $P=(z,y) \in D_\mu$, with $0< \mu < 1$, then $\tilde z \geq
1-\mu^2 >0$. It follows by the bounds of Lemma \ref{lemma-boundsA} and
the bound $0 < \lambda \leq H \leq \lambda^{-1}$,  that
\eqref{eqn-qr} is uniformly elliptic  on $D_\mu$.
Hence,  by the known results on the regularity of solutions to   strictly elliptic fully-nonlinear equations (see in \cite{CC}), one obtains  uniform $C^\infty$ bounds for $q^r$
on $D_\mu$, in terms of $\|q^r\|_{L^\infty(D_{\mu_0})}$, for
any $0< \mu < \mu_0 <1$. Notice that, in addition,
$\|q^r\|_{L^\infty(D_{\mu_0})}$ is uniformly bounded,  since $q_z$ is
bounded in $\B_\eta$. The above discussion leads to the following lemma:

\begin{lem}\label{lemma-infty}
For any $0 < \mu_0 <1$, there exists a constant $C(\mu_0)$ depending also on $ \|g\|_{C^2_{\partial \Omega}}, \|h\|_{C^2}, p, \lambda$ and $\rho$, such that
$$\|q^r\|_{C^\infty_s(D_\mu)} \leq C(\mu_0)$$
for all $0 < \mu < \mu_0$.
\end{lem}

One may now combine  Lemma \ref{lemma-infty} with  Lemma \ref{lemma-qzqy}
along the lines of the proof of Lemma 6.8 in \cite{DL1} to establish the $C^\alpha_s$ regularity of $z\, h_{zz}$ and $\sqrt{z} \, h_{zy}$, as stated next:

\begin{lem}\label{lemma-qzz}
Under the assumptions of Theorem \ref{theorem-c2a}, there exists a number $\alpha$ in $0 < \alpha <1$ and  constants $C$, $\eta$
depending only on the data $\|g\|_{C^2_{\partial \Omega}}, \|h\|_{C^2},$
$ p, \lambda, \rho\,$,  such that for any two points $Q_1=(z_1,y_1)$ and $Q_2=(z_2,y_2)$
in $\B_{\frac \eta 2}$, we have
$$
| z_1 q_{zz}( Q_1) - z_2 q_{zz}(Q_2)| +  | \sqrt{z_1}
q_{zz}( Q_1) - \sqrt{z_2}  q_{zz}(Q_2)| \leq C s(Q_1,Q_2)^\al.
$$
\end{lem}

Finally,  the  H\"older estimate for $q_{yy}$
can be derived from the H\"older estimates of $q_z, q_y$ and $z\, q_{zz},
\sqrt{z} \, q_{zy}$ and the regularity of $H$.

\begin{lem}\label{lemma-qyy}
Under the assumptions of Theorem \ref{theorem-c2a}, there exists a number $\alpha 
\in (0,1)$ and  constants $C$, $\eta$
depending only on the data $\|g\|_{C^2_{\partial \Omega}}, \|h\|_{C^2}$,
$p, \lambda, \rho\, $, such that for any two points $Q_1=(z_1,y_1)$ and $Q_2=(z_2,y_2)$
in $\B_{\frac \eta 2}$, we have
$$
|q_{yy}( Q_1) - q_{yy}(Q_2)| \leq s(Q_1, Q_2)^\alpha.
$$
\end{lem}   

Following an inductive argument as in Theorem 7.3 in  \cite{DH2}, we  can
show higher regularity, as stated next.

\begin{thm}\label{theorem-cka} Assume that $g \in C^{2,\alpha}_s$ is a solution of \eqref{eqn-g2}
which also satisfies the  assumptions of Theorem \ref{theorem-c2a}
and the additional assumption that $h \in C^{k+2}(\Omega)$,   there exist  constants $0 < \alpha < 1$,
$C<\infty$ and $\eta >0$, depending only on the data $\|g\|_{C^2_{\partial \Omega}}, \|h\|_{C^{k+2}}, p, \lambda, \rho$,  such that  for any free-boundary point
$P_0=(x_0,y_0)$, satisfying condition
\eqref{nvec2},  the function $x=q(z,y)$ satisfies
the estimate
$$\|q\|_{C^{k+\alpha}_s(\B_{\eta})} \leq C$$
on  $\B_\eta= \{ \, 0\leq z \leq \eta^2
, \,\, |y-y_0| \leq \eta \,\, \}$ for any positive integer $k$.
\end{thm}

\smallskip

We are  now in position to give the proof of Theorem  \ref{theorem-c-smooth}.
\smallskip

\noindent {\em Proof of Theorem \ref{theorem-c-smooth}}. 
Let $\eta$ denote the uniform constant  in Theorem \ref{theorem-cka}. Consider the
sub-domains
$$ \Omega_{\eta}^*(g)=\{x\in\Omega (g)|\,\, d(x,\Gamma (g))>\eta\}
\quad \mbox{and} \quad  \Omega_{\eta}(g)=\{x\in\Omega (g)|\,\, d(x,\Gamma (g) <\eta\}.$$
The estimate in Theorem \ref{theorem-cka} implies the bound 
$$\|g\|_{C^{k+\alpha}_s(\overline{\Omega_\eta(g)} )} \leq 
C(\|g\|_{C^2_{\partial \Omega}}, \|h\|_{C^2}, p, \lambda, \rho).$$
It remains to show that $g \in C^\infty(\Omega_{\eta}^*(g))$. Indeed, on $\Omega_{\eta}^*(g)$ we have 
$$0<\delta_{0} (\eta)\leq \det D^{2}f =h\, f^{p} \leq C( \lambda,  \max_{\partial\Omega}\vp)$$
for a positive constants $\delta_0$ and $C( \lambda,  \max_{\partial\Omega}\vp)$. 
Hence, $f$ satisfies a  Monge-Amp\'ere equation  as those  considered in  \cite{CKN}.

The bounds in Corollary \ref{cor-bii}  imply  the upper bound on any second derivative $f_{ii}$ on $\Omega_{\eta}^*(g)$, 
and   the lower bound of $f_{ii}$ follows from  the balance of the second derivatives $\det (D^{2}f) \approx 1$ on $\Omega_\eta^*$. Therefore $f$ satisfies a  uniformly elliptic 
equation and $\det^{1/2}(D^{2}f)$ is a concave operator.
Hence, the  $C^\infty$ regularity of  $f$,  satisfying $\det D^2 f = h\, f^p$, 
on $\Omega_{\eta}^*(g)$  follows  from the regularity theory for uniformly convex or concave fully-nonlinear  operators ( \cite{CC} ).\qed

\smallskip

\noindent The {\em proof of Theorem \ref{theorem-c2a} } readily follows from Lemmas \ref{lemma-qzqy}, \ref{lemma-qzz} and \ref{lemma-qyy}.

%%%%%%%%%%%%%%%%%%%%%%%%%%%%%%%%%%%%%%%%%%%%%%%%%%%%%%%%%%%%%%%%%%%%%%%%%%%%%%%%%%%
%%%%%%%%%%%%%%%%%%%%%%%%%%%%%%%%%%%%%%%%%%%%%%%%%%%%%%%%%%%%%%%%%%%%%%%%%%%%%%%%%%%

\section{Stability: $I$ is open}
\label{sec-stability}
In this section, we will utilize the  estimates of previous sections to show  the following stability  of solutions of \eqref{eqn-mapt} in the parameter $t$. This will conclude the proof of the
Theorem \ref{theorem-main}, as discussed in section  \ref{sec-classic2}.  

\begin{thm} Assume that $g_0$ is a classical  solution of \eqref{eqn-mapt} for $t=t_0$, satisfying conditions (H-1)--(H-4) and such that  $\|g_0\|_{C_s^{2,\alpha}}\leq 
C(\|\vp\|_{C^2_{\partial \Omega}}, p, \lambda, \rho\, )$.
Then, there is a $\delta>0$ such that for any $t$ with  $|t-t_0|<\delta$, the problem 
 \eqref{eqn-mapt} admits  a $C_s^{2,\alpha}$-solution $g(\cdot , t)$.
\label{theorem-stability}
\end{thm}

We will use the corresponding elliptic argument to the parabolic one which was used  in section 8 of  \cite{DH2}. Since the two arguments are quite similar,  we will only outline the proofs, referring 
the reader to \cite{DH2} (see also in \cite{DH1}) for  the details. 

 We pick a smooth surface ${\mathcal S} $, sufficiently close to the  $f_0=(q^{-2/3}\, g_0)^q$, such that its inner boundary $\partial {\mathcal S}$ lies on the $z=0$ plane and its outer boundary is $\partial\Omega$.
Denoting  by ${\mathcal D}$ a ring
    $${\mathcal D}= \{(u,v)\in \R^2: 1\leq u^2+v^2\leq 2\}$$
we let $S:{\mathcal D}\ra\R^2$ be a smooth parameterization for
the surface ${\mathcal S}$ which maps $\partial^{in}{\mathcal
D}=\{(u,v):u^2+v^2=1\}$ to $\partial^{in}{\mathcal S}={\mathcal
S}\cap\{z=0\}$ and $\partial^{out}{\mathcal
D}=\{(u,v):u^2+v^2=2\}$ to $\partial^{out}{\mathcal
S}=\partial\Omega$.
 We can find a smooth vector vector field
$$ {\mathcal T}=\begin{pmatrix} \T_1\\ \T_2 \\ T_3\end{pmatrix}    $$
which is transverse to the surface ${\mathcal S}\cap\{ z\geq \delta\}$
while it is pararell to the $z=0$ plane when $0\leq z\leq\delta$. Now we define the change of coordinate $\vp:{\mathcal D}\ra \R^3$ by
$$\begin{pmatrix} x\\y\\z\\\end{pmatrix}=\vp\begin{pmatrix} u\\v\\w \end{pmatrix}=S\begin{pmatrix}u\\v \end{pmatrix}+w{\mathcal T}\begin{pmatrix} u\\v \end{pmatrix}.$$

Via this coordinate change, the solution $z=f(x,y;t)$ of (MAt) will be mapped onto the graph
$$\left\{ \begin{pmatrix} u \\ v \\ w(u,v;t) \end{pmatrix} :\begin{pmatrix} u \\ v \end{pmatrix}
\right\}$$
if $z=f(x,y;t)$ is close to the surface ${\mathcal S}$.
By the choice of the parameterization $S$ of ${\mathcal S}$, we have $$(u,v)\in\partial^{in} {\mathcal D}\quad\text{iff}\quad z=0.$$
In the other words, the interfaces $\Gamma (g(x,y;t))=\partial\{(x,y):g(x,y;t)>0\}$
will be always mapped to the fixed boundary $\partial^{in}{\mathcal D}$.

\begin{definition}
We say $g(x,y;t)$ is of class $C^{k,2+\alpha}_s$ if the function $w(x,y;t)$ belongs to the class $C^{k,2+\alpha}_s({\mathcal D})$. Finally, we say that $g(x,y;t)$ are smooth up to the interface $\Gamma (g(x,y;t))$ if $w(u,v;t)$ is smooth on ${\mathcal D}$.
\end{definition}

 In addition, the equation \eqref{eqn-mapt} will be transformed to the boundary value problems
\begin{equation}
\begin{cases}
Mw(u,v;t)=0\quad &(u,v)\in {\mathcal D}\\
w(u,v;t)=\psi (u,v)\quad &(u,v)\in \partial^{out}{\mathcal D}
\end{cases}
\end{equation}
where $\psi(u,v)$ is the function,  uniquely determined by $\vp(x,y)$,  after the change of variables
and $Mw=F(D^2 w,Dw,w,u,v;t)$ is a fully nonlinear equation whose linearized equation at $t=0$
has the form \eqref{eq-linear} satisfying \eqref{eqn-matrix5},\eqref{eqn-coeffb}.

Theorem \ref{theorem-stability} follows  by combining  Theorem 8.4 in \cite{DH2}  and Theorems \ref{theorem-c2a} and  \ref{theorem-cka}.

 {\bf Acknowlegement:}
Panagiota Daskalopoulos is partially supported by the NSF grant
DMS-0604657. 
Ki-Ahm Lee was supported by the Korea Science and Engineering Foundation(KOSEF) grant funded by the Korea government(MOST) (No. R01-2006-000-10415-0  )


\end{document}